\documentclass[a4paper,11pt]{amsart}
\usepackage{latexsym,amscd,amssymb,url}
\usepackage{enumitem}
\usepackage{mathtools}
\usepackage{comment}

\usepackage[letterpaper]{geometry}

\usepackage[all,cmtip]{xy}  
\usepackage{graphicx}
\usepackage{xcolor}
\definecolor{dblue}{rgb}{0,0,.6}
\usepackage[colorlinks=true, linkcolor=dblue, citecolor=dblue, filecolor = dblue, menucolor = dblue, urlcolor = dblue]{hyperref}
\usepackage{tikz}
\usetikzlibrary{cd}

\numberwithin{equation}{section}

\newtheorem{theorem}{Theorem}[section]

\theoremstyle{plain} 

\newtheorem{corollary}[theorem]{Corollary}
\newtheorem{definition}[theorem]{Definition}
\newtheorem{lemma}[theorem]{Lemma}
\newtheorem{proposition}[theorem]{Proposition}

\theoremstyle{definition} 

\newtheorem{example}[theorem]{Example}

\newtheorem{remark}[theorem]{Remark}

\newcommand{\del}{\partial}
\newcommand{\Z}{\mathbb Z}
\newcommand{\Q}{\mathbb Q}

\newcommand{\A}{\mathbb A}

\newcommand{\C}{\mathbb C}

\newcommand{\R}{\mathbb R}

\newcommand{\CP}{\mathbb P}

\newcommand{\im}{\operatorname{im}}

\newcommand{\Hom}{\operatorname{Hom}}

\newcommand{\id}{\operatorname{id}}

\newcommand{\Spec}{\operatorname{Spec}}

\newcommand{\pr}{\operatorname{pr}}

\newcommand{\CH}{\operatorname{CH}}

\newcommand{\sing}{\operatorname{sing}}

\newcommand{\cl}{\operatorname{cl}} 
 
\newcommand{\Frac}{\operatorname{Frac}}

\newcommand{\dashedlongrightarrow}{\xymatrix@1@=15pt{\ar@{-->}[r]&}}
\renewcommand{\longrightarrow}{\xymatrix@1@=15pt{\ar[r]&}}
\renewcommand{\mapsto}{\xymatrix@1@=15pt{\ar@{|->}[r]&}}
\renewcommand{\twoheadrightarrow}{\xymatrix@1@=15pt{\ar@{->>}[r]&}}
\newcommand{\hooklongrightarrow}{\xymatrix@1@=15pt{\ar@{^(->}[r]&}}
\newcommand{\congpf}{\xymatrix@1@=15pt{\ar[r]^-\sim&}}
\renewcommand{\cong}{\simeq}

\begin{document}   
\title[Rationality of hypersurfaces]{Rationality of hypersurfaces}

\author{Stefan Schreieder} 
\address{Leibniz University Hannover, Institute of Algebraic Geometry, Welfengarten 1, 30167 Hannover, Germany.}
\email{schreieder@math.uni-hannover.de}

\date{\today}
%


\maketitle

\begin{abstract} 
We survey recent developments on rationality problems for algebraic varieties, with a particular emphasis on cycle-theoretic and combinatorial methods and their applications to hypersurfaces. 
\end{abstract}

\section{Introduction} 
 
A variety $X$ of dimension $n$ over a field $k$ is called \emph{rational} if it is birational to projective space $\mathbb P^n_k$; equivalently, its function field $k(X)$ is a purely transcendental extension of $k$.  
This means that away from a lower-dimensional subset, the solutions of the defining equations of $X$ can be parametrized bijectively by rational functions in $k(t_1,\dots ,t_n)$. 
The rationality problem asks whether a given variety $X$ is rational or not.
This is a classical problem in algebraic geometry, which boils down to a basic solvability property for a given set of algebraic equations.

There are several weaker versions of this, including 
\emph{unirationality}, \emph{retract rationality}, and \emph{stable rationality}.  
A variety $X$ is unirational if there is a dominant rational map $g\colon \mathbb P^N_k \dashrightarrow X$; over algebraically closed fields this still ensures that almost all solutions of the defining equations can be parametrized by rational functions (though not uniquely) but this typically fails over non-closed fields.  
We say that $X$ is retract rational if there is a unirational parametrization $g$ with a rational section.
Equivalently, 
there are rational functions which, for all field extensions $L/k$,  allow to parametrize all $L$-rational points of $X$ outside a lower dimensional subset, see Definition \ref{def:retract}. 
Finally, $X$ is stably rational if $X \times \mathbb P^m$ is rational for some $m \ge 0$.  
One then has the following sequence of implications
\[
\text{rational} \;\Longrightarrow\; \text{stably rational} \;\Longrightarrow\; \text{retract rational} \;\Longrightarrow\; \text{unirational},
\]
with strictness of the first and last implications proven in \cite{BCTSS} and \cite{AM}, respectively. 
Strictness in the middle is open over algebraically closed fields, but established over some non-closed fields in \cite{endo-miyata}.

A natural and historically central testing ground for these questions is the case of smooth hypersurfaces $X \subset \mathbb P^{n+1}_k$ of degree $d$, cf.\ \cite{kollar-survey}.  
If $d \geq n+2$, then $K_X=\mathcal O_X(d-n-2)$ has nontrivial sections, so $X$ is far from rational (it is not even separably rationally connected).  
The interesting range is $d \leq n+1$, where $X$ is Fano and hence rationally chain connected \cite{campana,KMM}, i.e.\ it contains many rational curves and hence looks like a rational variety in some sense.
Moreover, if $k=\C$ and $d!\leq \log_2n$, then $X$ is unirational, see \cite{HMP,BR}.

To establish that a variety satisfies one of the above rationality properties, one typically has to provide an explicit construction.  
Conversely, to prove failure of such a property, one must exhibit a nontrivial obstruction.  
Moreover, if the base change of a variety $X$ to the algebraic closure of $k$ is not (stably/retract) rational, then neither is $X$, nor the base change $X_L$ to any field extension $L/k$.  
Thus the strongest irrationality results are obtained over algebraically closed fields; in characteristic zero, this allows us to work over the ground field $\C$.

For smooth hypersurfaces $X \subset \mathbb P^{n+1}_\C$ with $d \leq n+1$, the picture is as follows. 
For $d\leq 2$, such hypersurfaces are always rational and the result goes back to the ancient Greeks.
In 1866, Clebsch showed that smooth cubic surfaces over $\C$ are rational, settling the case $n=2$. 
In the 1970s, Clemens--Griffiths \cite{clemens-griffiths} and Iskovskikh--Manin \cite{IM} showed that cubic and quartic threefolds are not rational, via intermediate Jacobians and birational rigidity, respectively.  
Birational rigidity was later developed further and culminated in de Fernex’s result \cite{deF1} that smooth hypersurfaces of degree $d=n+1$ are irrational in all dimensions $n\geq 3$ (see also \cite{Pu1,Pu2}).  
A different approach, using mixed characteristic degenerations and differential forms in positive characteristic, was introduced by Koll\'ar \cite{kollar} in 1995, who proved that very general hypersurfaces of degree $d \geq 2\lceil (n+3)/3 \rceil$ are not ruled, hence not rational.  
This settled, for very general hypersurfaces, roughly one third of the remaining open cases.

In the last decade, exciting further progress has been made in this area.  
A distinctive feature of the subject is the variety of methods that have been used successfully by different sets of authors. 
These include: 
\begin{enumerate}[label={(\arabic*)}]
\item algebraic cycles, see \cite{voisin,CTP,Sch-duke}, \cite{voisin-JEMS,shen}, and \cite{Pavic-Sch,Lange-Sch};\label{item:intro:alg-cycles}
\item unramified cohomology and quadratic forms, see \cite{CTO,HPT,Sch-JAMS};\label{item:intro:unramified}
\item motivic integration and weak factorization, see \cite{NS,KT};\label{item:intro:motivic-integration}
\item tropical and discrete geometry, see \cite{NO,moe};\label{item:intro:tropical-discrete-geo}
\item the combinatorial theory of regular matroids, see \cite{gwena,EGFS};\label{item:intro:matroids}
\item mirror symmetry and quantum cohomology, see \cite{KKPY}.\label{item:intro:quantum-cohomology}
\end{enumerate}

We survey here the irrationality results for hypersurfaces obtained via these methods.  
Totaro \cite{totaro} used the cycle-theoretic approach \ref{item:intro:alg-cycles} of Voisin \cite{voisin} and Colliot-Th\'el\`ene--Pirutka \cite{CTP}, to upgrade Koll\'ar’s earlier result \cite{kollar} from rationality to stable and retract rationality, giving the linear bound $d \geq 2\lceil (n+2)/3 \rceil$. 
Using the improvement from \cite{Sch-duke} and combining it with \ref{item:intro:unramified}, we obtained 
a logarithmic bound as follows.

\begin{theorem}[\cite{Sch-JAMS}] \label{thm:JAMS-intro}
A very general hypersurface $X\subset \CP^{n+1}_k$ of dimension $n\geq 3$ and degree $d\geq (\log_2 n)+2$ over a field of characteristic different from $2$ is not retract rational.
\end{theorem}

The above logarithmic bound is equivalent to $n \leq 2^{d-2}$.  
Using ideas from motivic integration \cite{NS,KT} together with tropical and discrete methods, Nicaise--Ottem \cite{NO} improved the stable rationality bound for $k=\C$ and $n=5$ from $d\geq 5$ to $d\geq 4$.  
Moe \cite{moe} extended this to $n \leq (d+1)2^{d-4}$, which recovers precisely the Nicaise--Ottem bound for $d=4$; his proof uses Theorem \ref{thm:JAMS-intro} as input.  
Joint work with Lange \cite{Lange-Sch} generalized Moe’s theorem using a variant of the cycle-theoretic method from \cite{Pavic-Sch}, yielding the following; an alternative proof in characteristic zero is due to Hotchkiss--Stapleton \cite{hotchkiss-stapleton}.

\begin{theorem}[\cite{NO,moe,Pavic-Sch,Lange-Sch}]\label{thm:Lange-Sch-intro} 
A very general hypersurface $X\subset \CP^{n+1}_k$ of degree $d\geq 4$ and dimension $n \leq (d+1)2^{d-4}$ over a field of characteristic different from $2$ is not retract rational. 
\end{theorem}

The above techniques yield strong asymptotic results in large dimension and degree $\geq 4$, but they do not apply to cubics. 
After the work of Clemens--Griffiths \cite{clemens-griffiths} in the 1970s, the first open case is the stable or retract rationality of cubic threefolds.  
In joint work with Engel and de Gaay Fortman, we used matroid theory to show that an obstruction of Voisin \cite{voisin-JEMS} in terms of algebraicity of the minimal class of the intermediate Jacobian is nontrivial for very general cubic threefolds, thus settling this longstanding open case as well:

\begin{theorem}[\cite{EGFS}] \label{thm:EGFS-intro}
A very general cubic threefold $X\subset \CP^4_\C$ is not retract rational. 
\end{theorem}

In all dimensions $n\geq 4$, taking the cone over such a threefold produces irrational singular Fano hypersurfaces of degree $3$ with only terminal singularities.
These examples, however, do not shed light on smooth cubics in higher dimensions, cf.\ \cite[Section 4]{totaro}.
For instance, even in dimension $4$ rationality of smooth cubics was an open problem until very recently, when Katzarkov--Kontsevich--Pantev--Yu proved the following striking result: 

\begin{theorem}[\cite{KKPY}] \label{thm:KKPY}
A very general cubic fourfold $X\subset \CP^5_\C$ is not rational.
\end{theorem}

The aforementioned results on irrationality for (very general) hypersurfaces cover essentially all results known up till now.  
The next open cases concern stable rationality of cubic fourfolds and rationality of cubic fivefolds.
In degree $d\geq 4$, the next open cases are quartic sixfolds and quintics in dimension $13$, cf.\ \cite[Theorem 7.1]{Lange-Sch}.  
 
This survey is organized as follows.  
In Section \ref{sec:rationality} we give examples of rational and retract rational varieties, and explain how to obstruct retract rationality via decompositions of the diagonal \cite{ACTP,voisin}.  
Section \ref{sec:deg+coho} sketches the proof of Theorem \ref{thm:JAMS-intro} (see the survey \cite{Sch-survey} for more details), while Section \ref{sec:motivic} explains the improvements leading to Theorem \ref{thm:Lange-Sch-intro}.  
Sections \ref{subsec:coho-dec-diag-1}--\ref{subsec:coho-dec-diag-2} revisit the aspects of Voisin’s work \cite{voisin-JEMS,voisin-min-class} that are used in the proof of Theorem \ref{thm:EGFS-intro}; we make some effort to formulate parts of it over arbitrary fields. 
Section \ref{subsec:matroids} recalls regular matroids, and Section \ref{subsec:cubic-threefolds} comments on the proof strategy of Theorem \ref{thm:EGFS-intro}.
 
Due to limitations of space, we will not try to address the proof of Theorem \ref{thm:KKPY}.  
We also omit many further exciting developments, such as the rationality problem over non-closed fields, 
equivariant birational geometry, Noether's problem, relations to the Cremona group, birational rigidity or derived category approaches to rationality.  
Excellent surveys on various aspects of the rationality problem that we do not discuss in detail here can for instance be found in \cite{CT-Sansuc,kuznetsov,hassett-survey,pirutka,bogomolov-tschinkel,CT-garda,voisin-survey,kollar-survey,kollar-survey-2,debarre}.

\section{Basic notions} \label{sec:rationality}

\subsection{Examples of rational varieties.} 
The following example goes back to the ancient Greeks.

\begin{example}[Stereographic projection]
Let $X\subset \CP^{n+1}_k$ be a quadric with a $k$-rational point $p\in X(k)$ in its smooth locus.
Choose a hyperplane $H\subset \CP^{n+1}_k$ not containing $p$, and let $L_{p,q}$ be the line through $p$ and $q\neq p$.
The rational map $\varphi\colon X\dashrightarrow H$, $q\mapsto L_{p,q}\cap H$, is birational because a general line through $p$ meets $X$ in precisely two points. 
Hence $X$ is rational.
\end{example}

The inverse $\varphi^{-1}$ sends a general point $q\in H$ to the residual point $r$ of $L_{p,q}\cap X=\{p,r\}$, giving an explicit parametrization of $X(k)$ outside the indeterminacy loci.  
For instance, if $X\subset \CP^2_k$ is given by $x^2+y^2=z^2$, one can recover the parametrization $x=2ts$, $y=s^2-t^2$, $z=s^2+t^2$ with $[s:t]\in \CP^1_k$.  
For $k=\Q$, this yields the classical Pythagorean triples 
$a=2mn$, $b=m^2-n^2$, $c=m^2+n^2$ for coprime $m,n\in \Z$ satisfying $a^2+b^2=c^2$. 

\begin{example}[Euler 1761]
Rationality of the Fermat cubic surface $w^{3}+x^{3}+y^{3}+z^{3}=0$ over $\Q$ was shown by Euler.
An explicit description of all rational solutions in some dense open subset is given by
\small{
\[
(w,x,y,z)=\bigl(\lambda(1-(a-3b)(a^2+3b^2)),\;\lambda((a+3b)(a^2+3b^2)-1),\;-\lambda((a+3b)-(a^2+3b^2)^2),\;-\lambda((a^2+3b^2)^2-(a-3b))\bigr),
\]}
where $a,b,\lambda\in\mathbb{Q}$, see \cite[\S 13.7]{hardy-wright}.
\end{example}

 
\begin{example}[Clebsch 1866] \label{ex:cubic-surface}
Let $X\subset \CP^3_k$ be a smooth cubic surface over an algebraically closed field $k$.
It can be shown that $X$ contains two lines $L_1,L_2\subset X$ with $L_1\cap L_2=\emptyset$.
Then the rational map $L_1\times L_2\dashrightarrow X$, $(p,q)\mapsto r$, where $r$ is the residual intersection point of $X\cap L_{p,q}$, is birational. Hence $X$ is rational.
\end{example}

The first irrational smooth cubic surfaces over non-closed fields were produced by Segre, see e.g.\ \cite{segre}. 


\begin{example}
While the very general cubic fourfold $X\subset \CP^{5}_\C$ is irrational \cite{KKPY}, many special smooth cubic fourfolds are rational. 
The simplest case is when $X$ contains two disjoint planes, so that a similar construction as in Example \ref{ex:cubic-surface} applies.
Much more sophisticated constructions are 
contained in the recent work of Russo and Staglian\`o  \cite{russo-stagliano-1,russo-stagliano-2}; see also Hassett's survey \cite{hassett-survey} on the subject.
\end{example}

\subsection{Retract rational varieties.} 
A $k$-variety $X$ is \emph{unirational} if there exists a dominant rational map $g\colon \CP^n_k\dashrightarrow X$.
For example, building on work of Segre and Manin, Koll\'ar \cite{kollar-cubics} proved that a smooth cubic hypersurface of dimension $n\geq 2$ over any field $k$ is unirational precisely when it contains a $k$-point.
Unirationality implies that a Zariski-dense set of $k$-points can be parametrized by rational functions, but this does in general not guarantee that we can parametrize ``almost all'' of them in a satisfactory way.
For instance, the unirational real cubic surface 
$S = \{\lambda x^3 + y^3 + z^3 + w^3 + (-x-y-z-w)^3=0\}\subset \CP^3_\R$ with $\tfrac{1}{16}<\lambda<\tfrac{1}{4}$ has disconnected real locus $S(\R)$, cf.\ \cite{segre} and \cite[Example 3.1]{polo-blanco-top}.
However, any dominant rational map $g\colon \CP^N_\R\dashrightarrow S$ can be resolved by a sequence of smooth blow-ups and since any such modification of projective space has connected real locus, $g$ can parametrize $\R$-points on at most one component of $S(\R)$.
This motivates the following; for Saltman's original motivation in \cite{saltman-2}, see \cite[Proposition 1.2, Remark 1.3 and Proposition 3.15]{CT-Sansuc}:

\begin{definition} \label{def:retract}
    A variety $X$ over a field $k$ is retract rational if for some integer $N$ there is a commutative diagram of rational maps such that the composition $g\circ f$ is defined:
\begin{equation} \label{diag:retract-rational}
\vcenter{\xymatrix{
X\ar@{-->}[dr]_f  &   & \ar@{-->}[ll]^{\id_X} X\\
& \CP_k^N\ar@{-->}[ru]_g &
}}
\end{equation}
Equivalently, there are open dense subsets $U\subset X$ and $V\subset \CP^N$ and a morphism $g\colon V\to U$ which is surjective on $L$-rational points for all field extensions $L$ of $k$: $V(L)\twoheadrightarrow U(L)$ for all $L/k$.
\end{definition}

To see the above claimed equivalence, note that the former implies the latter, because \eqref{diag:retract-rational} implies that we can find open dense subsets $U\subset X$ and $V\subset \mathbb P^N$ such that $f|_U\colon U\to V$ and $g|_V\colon V\to U$ are morphisms with $g|_V\circ f|_U=\id_U$. Hence, $g|_V$ is surjective on $L$-rational points for all $L/k$. To see the converse, apply the latter definition to $L=k(X)$ and lift the diagonal point $\delta_X\in U(k(X))$ to a point $\delta'\in V(k(X))$. Then $\delta'$ corresponds to a rational map $f\colon X\dashrightarrow V\subset \mathbb P^N$ and the condition $g(\delta')=\delta_X$ implies $g\circ f=\id_X$, as we want.



\begin{example}
A stably rational variety is retract rational.
More generally, if $X,Y$ are $k$-varieties such that $X\times Y$ is rational and $Y$ contains a $k$-rational point, then $X$ is retract rational.
Indeed, under this assumption, $X\times Y$ serves the purpose of $\CP_k^N$ up to birational equivalence and the maps $f$ and $g$ are induced by the inclusion of a factor and the projection, respectively.
\end{example}

Note that in the diagram \eqref{diag:retract-rational}, the map $f$ is necessarily generically injective, i.e.\ $f$ induces an isomorphism between $\kappa(\eta_X)$ and $\kappa(f(\eta_X))$, where $\eta_X\in X$ denotes the generic point. 
The rational map $g$ is then a rational retraction of $f$, i.e.\ a rational map such that the composition $g\circ f$ is defined and coincides with the identity as a rational map.
It is worth noting the following, see \cite[Lemma 2.1]{Sch-MRL}.

\begin{lemma}
Let $k$ be a field and let $X$ be a retract rational $k$-variety. 
Then for any rational map  $f\colon X\dashrightarrow \CP^N_k$ which is generically injective, there is a
rational retraction $g\colon \CP^N\dashrightarrow X$. 
\end{lemma}

Hence, if a hypersurface $X\subset \CP^{n+1}_k$ is retract rational, then the inclusion $X\hookrightarrow \CP^{n+1}_k$ admits a rational retraction.
Since any variety is birational to a hypersurface, it follows that we may take $N=\dim X+1$ in \eqref{diag:retract-rational}. 

The following uniform result on retract rationality has recently been proven by Banecki.

\begin{theorem}[\cite{banecki}] \label{thm:banecki}
Let $X$ be a smooth variety over an infinite field $k$.
If $X$ is retract rational then it is uniformly retract rational, that is, for any point $x\in X$ there is a commutative diagram as in \eqref{diag:retract-rational} such that $f$ is defined at $x$ and $g$ is defined at $f(x)$.
\end{theorem} 

\begin{remark}
It follows that any smooth retract rational variety over an infinite field $k$ admits finitely many unirational parametrizations such that for any $L/k$ any $L$-rational point of $X$ is covered by one of them.
\end{remark}

\subsection{Functorial birational invariants.}

Morally, we would like to talk about a functor from the “category’’ of smooth projective $k$-varieties with rational maps between them to another category.  
Since rational maps need not be composable, such a category does not exist (see however \cite{kahn-sujata,hotchkiss-stapleton}). 
For our purposes the following adhoc notion suffices.

\begin{definition} \label{def:bir'l-invariant}
Let $k$ be a field.  
A \emph{functorial birational invariant} $F$ of smooth projective $k$-varieties with values in a category $\mathcal A$ consists of
\begin{itemize}
    \item an object $F(X)\in \mathcal A$ for each smooth projective $k$-variety $X$, and
    \item a morphism $F(f)\in \Hom_{\mathcal A}(F(X),F(Y))$ for each rational map $f\colon X\dashrightarrow Y$,
\end{itemize}
such that if $f\colon X\dashrightarrow Y$ and $g\colon Y\dashrightarrow Z$ are composable, then $F(g\circ f)=F(g)\circ F(f)$, and $F(\id_X)=\id_{F(X)}$ for all $X$.
\end{definition}

Recall that an object $0\in\mathcal A$ in a category is a  \emph{zero object} if $\Hom(0,A)=\Hom(A,0)=\{\ast\}$ for all $A\in \mathcal A$.

\begin{lemma} \label{lem:bir'l-invariant}
Let $F$ be a functorial birational invariant with values in a category $\mathcal A$ admitting a zero object $0$ and such that $F(\CP^N_k)=0$ for all $N\geq 0$.
If $X$ is retract rational, then $F(X)=0$. 
\end{lemma}

\begin{proof} 
From diagram \eqref{diag:retract-rational} we have rational maps $f,g$ with $g\circ f=\id_X$.  
Thus $F(g)\circ F(f)=\id_{F(X)}$.  
Conversely, $F(f)\circ F(g)$ is an endomorphism of $F(\CP^N_k)=0$, hence equals $\id_0$.  
So $F(f),F(g)$ are inverses and $F(X)\cong 0$, as claimed.
\end{proof}

The somewhat adhoc Definition \ref{def:bir'l-invariant} was introduced to bypass the fact that smooth projective varieties with rational maps do not form a category.  
An alternative approach is to consider the category ${\rm Corr}(k)$ of Chow correspondences (restricted to smooth projective integral schemes): objects are smooth projective $k$-varieties, and
\[
\Hom_{{\rm Corr}(k)}(X,Y):=\CH_{\dim X}(X\times Y).
\]
If $\Gamma\in \CH_{\dim X}(X\times Y)$ and $\Omega\in \CH_{\dim Y}(Y\times Z)$, then
\[
\Omega\circ \Gamma=(\pr_{XZ})_\ast\bigl(\pr_{XY}^\ast \Gamma \cdot \pr_{YZ}^\ast \Omega\bigr)\in \CH_{\dim X}(X\times Z),
\]
where $\pr_{XY},\pr_{YZ},\pr_{XZ}$ are the respective projections from $X\times Y\times Z$.  
The above intersection product may be realized via a moving lemma---an operation unavailable for rational maps directly.

If $f\colon X\dashrightarrow Y$ is a rational map, its graph $\Gamma_f$ defines a correspondence, hence a morphism in ${\rm Corr}(k)$, again denoted by $f$.  
Unlike rational maps, correspondences can always be composed, so any $f\colon X\dashrightarrow Y$ and $g\colon Y\dashrightarrow Z$ yield a well-defined $g\circ f\in \Hom_{{\rm Corr}(k)}(X,Z)$.
We then have the following

\begin{lemma} \label{lem:functor->bir'l-invariant}
Let $F\colon {\rm Corr}(k)\to \mathcal A$ be an additive functor from the category ${\rm Corr}(k)$ of Chow correspondences of $k$-varieties to an additive category $\mathcal A$.
Assume that for any correspondence $\Omega\in \CH_{\dim X}(X\times Y)$  
whose support does not dominate the first factor, 
the induced morphism $F(\Omega)\colon F(X)\to F(Y)$ is zero.  
Then the assignments $X\mapsto F(X)$ and $f\mapsto F(\Gamma_f)$ for a rational map $f$ define a functorial birational invariant.
\end{lemma}

\begin{proof}
Since $F$ is a functor, $F(\id_X)=\id_{F(X)}$.  
For composable rational maps $f,g$, we have to show that $F(g)\circ F(f)=F(g\circ f)$.
Since $F$ is a functor, we clearly have $F(\Gamma_g)\circ F(\Gamma_f)=F(\Gamma_g\circ \Gamma_f)$. 
However, since $f$ and $g$ are only rational maps, $\Gamma_{g\circ f}$ need not be equal to $\Gamma_g\circ \Gamma_f$.
Nonetheless, their difference
\[
\Omega:=\Gamma_{g\circ f}-\Gamma_g\circ \Gamma_f\in \CH_{\dim X}(X\times Z)
\]
can be represented by a cycle not dominating the first factor, hence $F(\Omega)=0$ by assumption.  
Thus
\[
F(g)\circ F(f)=F(\Gamma_g\circ \Gamma_f)=F(\Gamma_{g\circ f}-\Omega)=F(\Gamma_{g\circ f})=F(g\circ f),
\]
as claimed.
\end{proof}

\subsection{Zero-cycles and decompositions of the diagonal.} \label{subsection:CH0}

If $X$ is a smooth projective $k$-variety, we may consider the Chow group of $0$-cycles $\CH_0(X)$ modulo rational equivalence.
This defines a functor
\begin{align} \label{def:functor-CH0}
\CH_0\colon {\rm Corr}(k)\longrightarrow {\rm Ab}, \quad \quad  \CH_{\dim X}(X\times Y) \ni \Gamma\mapsto \Gamma_\ast\colon \CH_0(X)\to \CH_0(Y)
\end{align}
to the category of abelian groups, where $\Gamma_\ast(z)=p_{Y\ast}(\Gamma\cdot p_{X}^\ast (z))$ with $p_X\colon X\times Y\to X$ and $p_Y\colon X\times Y\to Y$. 

\begin{example} \label{ex:CH_0}
The functor $\CH_0$ from \eqref{def:functor-CH0} is a functorial birational invariant of smooth projective varieties over $k$ (see Definition \ref{def:bir'l-invariant}).
This follows from Lemma \ref{lem:functor->bir'l-invariant}: if $\Omega\in \CH_{\dim X}(X\times Y)$ is a correspondence that does not dominate $X$, then the induced map $\Omega_\ast\colon \CH_0(X)\to \CH_0(Y)$ is zero.
Indeed, via the moving lemma, any zero-cycle in $X$ can be represented by a cycle that does not meet $\pr_X(\Omega)\subset X$ and so $\Omega\cdot p_X^\ast (z)=0$. 
\end{example}

Note that $\CH_0$ does not satisfy the vanishing assumption on projective space needed to apply Lemma \ref{lem:bir'l-invariant}, since $\CH_0(\CP^N_k)\cong \Z$ is generated by the class of any $k$-point.  
To remedy this, consider the degree map
\[
\deg\colon \CH_0(X)\longrightarrow \Z,\quad \sum a_i[x_i]\mapsto \sum a_i \deg(\kappa(x_i)/k),
\]
and define the degree-zero part
\[
\CH_0(X)_0 \coloneqq \ker(\deg).
\]
The degree is preserved under correspondences, so $\CH_0(-)_0$ yields, as above, a functorial birational invariant for smooth projective $k$-varieties.  
This invariant vanishes on projective space and thus obstructs retract rationality by Lemma \ref{lem:bir'l-invariant}.  
On the other hand, it is too weak over algebraically closed fields $k=\bar k$: if $X$ is rationally chain connected over $k=\bar k$ (in particular, if $X$ is Fano, see \cite{campana,KMM}), then $\CH_0(X)_0=0$; this holds in particular for smooth hypersurfaces $X\subset \CP^{n+1}_k$ of degree $d\leq n+1$. 
Rational (chain) connectedness is a geometric property: for any algebraically closed extension $L/k$, any two closed points of $X_L$ are connected by a chain of rational curves.
Allowing non-closed field extensions therefore adds nontrivial additional information: 

\begin{lemma} \label{lem:CH_0(X)_0}
Let $L/k$ be a field extension.  
Then $X\mapsto \CH_0(X_L)_0$ defines a functorial birational invariant which vanishes on projective space.
\end{lemma}
\begin{proof}
As in Example \ref{ex:CH_0}, because $\deg(\Gamma_\ast z)=\deg z$ for $z\in \CH_0(X)$ and $\Gamma\in \CH_{\dim X}(X\times Y)$.
\end{proof}

If $X$ is retract rational, Lemmas \ref{lem:bir'l-invariant} and \ref{lem:CH_0(X)_0} imply $\CH_0(X_L)_0=0$ for all $L/k$.  
This motivates: 

\begin{definition}[{\cite[\S1.2]{ACTP}}] \label{def:CH0-trivial}
A projective variety $X$ over $k$ is universally $\CH_0$-trivial if, for all $L/k$, the degree map $\deg\colon \CH_0(X_L)\to \Z$ is an isomorphism.  
Equivalently, $X$ has a zero-cycle of degree $1$ and $\CH_0(X_L)_0=0$ for all $L/k$.
\end{definition}

\begin{lemma} \label{lem:retract-rational->CH0-trivial}
If $X$ is a smooth projective retract rational variety over $k$, then it is universally $\CH_0$-trivial.
\end{lemma}
\begin{proof}
For any field extension $L/k$, we have $\CH_0(X_L)_0=0$ by Lemmas \ref{lem:bir'l-invariant} and \ref{lem:CH_0(X)_0}.  
A zero-cycle of degree $1$ is obtained by $(\Gamma_g)_\ast z\in \CH_0(X)$ for some $z\in \CH_0(\CP^N)$ of degree $1$. 
(In fact, $X$ contains a $k$-point because it is unirational via the map $g$ in \eqref{diag:retract-rational}; the rational point arises from the valuative criterion of properness by restricting $g$ to a smooth curve in $\CP^N$ which meets the locus of definition of $g$ and contains a $k$-rational point.)
\end{proof}

The following notion goes back to Bloch \cite{bloch} and Bloch--Srinivas \cite{BS}, and has for instance been studied in \cite{ACTP} and \cite{voisin}.

\begin{definition}
A variety $X$ over a field $k$ admits a \emph{decomposition of the diagonal} if there exist a zero-cycle $z\in \CH_0(X)$ and a cycle $\Gamma\in \CH_{\dim X}(X\times X)$ not dominating the first factor such that
\begin{equation} \label{eq:diagonal-decomposition}
\Delta_X = X\times z + \Gamma \in \CH_{\dim X}(X\times X).
\end{equation}
\end{definition}

Using the localization sequence \cite[Prop.~1.8]{fulton} and a limit argument (see \cite[Lemma~7.3]{Sch-survey}), this is equivalent to requiring that for some $z\in \CH_0(X)$,
\begin{equation} \label{eq:diagonal-decomposition-zero-cycle}
\delta_X = z_{k(X)} \in \CH_0(X_{k(X)}),
\end{equation}
where $\delta_X$ denotes the $k(X)$-rational point of $X_{k(X)}$ induced by the diagonal and $z_{k(X)}$ is the base change of $z$.  

\begin{proposition}\label{prop:diagonal}
For a smooth projective variety $X$ over a field $k$, the following are equivalent:
\begin{enumerate}[label={(\arabic*)}]
    \item $X$ admits a decomposition of the diagonal;\label{item:prop:diagonal:1}
    \item $X$ is universally $\CH_0$-trivial;\label{item:prop:diagonal:2}
    \item $\CH_0(X_{k(X)})_0=0$ and $X$ has a zero-cycle of degree $1$.\label{item:prop:diagonal:3}
\end{enumerate}
\end{proposition}
\begin{proof}
Assume first that $X$ admits a decomposition of the diagonal as in \eqref{eq:diagonal-decomposition}.
Base change to an algebraic closure of $k$ and intersection with the second factor shows that $\deg z=1$. 
Hence, $X$ contains a zero-cycle of degree $1$.
We then consider the action of the diagonal $\Delta_X$ on $\CH_0(X_L)$ and note that $\Gamma_\ast$ acts trivially, while $X\times z$ sends a zero-cycle $\alpha$ to $(\deg \alpha)\cdot z$.
Since the diagonal acts as the identity, we find that $\CH_0(X_L)=\Z\cdot z_L$ for all $L/k$ and so $X$ is universally $\CH_0$-trivial.
Hence,  \ref{item:prop:diagonal:1} $\Rightarrow$ \ref{item:prop:diagonal:2}.
The implication  \ref{item:prop:diagonal:2} $\Rightarrow$ \ref{item:prop:diagonal:3} is trivial.
Finally, assume \ref{item:prop:diagonal:3} and let $z\in \CH_0(X)$ be a zero-cycle of degree $1$.
Then $\delta_X-z_{k(X)}$ is a zero-cycle of degree zero and hence vanishes by assumptions.
This proves \eqref{eq:diagonal-decomposition-zero-cycle}, which is equivalent to a decomposition as in \eqref{eq:diagonal-decomposition}. 
\end{proof}

\subsection{Very general hypersurfaces.}
A hypersurface $X \subset \CP^{n+1}_k$ over a field $k$ is called \emph{very general} if $X$ is isomorphic to a hypersurface $\{F=0\}\subset \CP^{n+1}_k$ cut out by a polynomial $F=\sum_{|I|=d} a_Ix^I$ whose coefficients $a_I$  are algebraically independent over the prime field of $k$.
We also say that $X$ is very general if its base change to some field extension $L/k$ satisfies this property.
Very general hypersurfaces exist over $k$ if it has sufficiently large transcendence degree over its prime field.
For example, the elliptic curve $zy^2 = x^3 + z^2x + tz^3$ over $\Q(t)$ is very general because a plane cubic curve over $\C$ whose 10 coefficients are algebraically independent over $\Q$ is isomorphic to a cubic curve as above for some transcendental element $t\in \C$.

There is an abstract version of this notion.  
Let $\pi\colon \mathcal X \to B$ be a proper morphism of varieties over $k$, and let $k_0 \subset k$ be a subfield such that $\pi$ admits a model $\pi_0\colon \mathcal X_0 \to B_0$ over $k_0$.  
Choose $k_0$ to be finitely generated over its prime field and of minimal transcendence degree among all such fields of definition.  
A fiber of $\pi$ at a point $b\in B(k)$ is very general if $ b $ maps to the generic point under the projection $B=B_0\times_{k_0}k \to B_0$.  
Equivalently, there is an inclusion $k_0(B_0) \subset k$ of fields which identifies $ X_b $ with the base change of the generic fiber $\mathcal X_0 \times_{B_0} k_0(B_0)$ of $ \pi_0 $.  
If $ k $ is algebraically closed, this means that $ X_b $ is abstractly isomorphic to (a base change of) the geometric generic fiber of $\pi_0$, and hence, up to a base change, to that of $\pi$.  
In particular, if $k=\bar k$, then $X_b$ specializes to any closed fiber of $\pi$ (cf.\ \cite[\S2.2]{Sch-duke}). 
The locus of points $b\in B(k)$ such that the corresponding fiber is very general corresponds to the $k$-points of the complement of a countable union of proper closed subsets of $B$.

If $B$ is a fine moduli space or a component of the Hilbert scheme that parametrizes a certain type of projective varieties, e.g.\ hypersurfaces of degree $d$ in $\CP^{n+1}_k$, with universal family $\pi\colon \mathcal X\to B$, then we say that such a variety $X$ over $k$ is very general if up to some base change to a larger field, $X$ is isomorphic to a very general fiber of $\pi$.
Applied to the universal family of hypersurfaces over $\CP(H^0(\mathbb P^{n+1},\mathcal O(d)))$,  this recovers the above ad hoc definition of very general hypersurfaces.
(As in the case of elliptic curves, the redundancy in Hilbert schemes compared to moduli spaces is resolved by allowing base changes of $X$ to larger fields.)
    
In the above discussion, several base changes were allowed.
To justify this, note that if $X$ is a variety over $k$ that is rational, stably rational, or retract rational, then the same holds after any field extension $L/k$.  
Conversely, if $X_L$ has this property and $k$ is algebraically closed, then $X$ itself does.  
Indeed, one may reduce to the case where $L/k$ is finitely generated, hence the function field of a $k$-variety $Z$, over which the relevant rational maps spread out; specializing at a general $k$-rational point of $Z$ gives the claim.  

Since rationality is a closed property in smooth projective families (see \cite{NS,KT}), we obtain the following:

\begin{theorem}[\cite{NS,KT}]
If a very general hypersurface of degree $d$ and dimension $n$ over some field of characteristic zero is (stably) rational, then every smooth hypersurface of degree $d$ and dimension $n$ over any algebraically closed field of characteristic zero is (stably) rational.
\end{theorem}

\section{Degeneration and unramified cohomology} \label{sec:deg+coho}
\subsection{Degeneration techniques.} 

In 2015, Voisin \cite{voisin} introduced cycle-theoretic methods to prove stable irrationality for new classes of varieties. 
She showed that very general quartic double solids $X \to \CP^3_\C$ do not admit a decomposition of the diagonal, hence are not retract rational. 
The key idea is a degeneration argument: $X$ specializes to a singular quartic double solid $Y \to \CP^3_\C$, birational to the Artin--Mumford example \cite{AM}. 
A resolution $Y'$ of $Y$ has torsion in $H^3(Y',\Z)$, which rules out a decomposition of the diagonal by analyzing the action of \eqref{eq:diagonal-decomposition} on cohomology. 
Since decomposability of the diagonal specializes well in algebraic families, $X$ would admit a decomposition only if $Y$ did. 
Careful analysis of the singularities of $Y$ then yields a contradiction.  

The passage from $Y$ to $Y'$ is subtle but essential: for instance, a smooth cubic surface is rational and specializes to a cone over an elliptic curve; while the cone admits a decomposition of the diagonal, its resolution does not.  

Colliot-Th\'el\`ene--Pirutka \cite{CTP} extended Voisin’s method to more singular degenerations over arbitrary discrete valuation rings, using Fulton's specialization map on Chow groups \cite[\S 20.3]{fulton}. 
The distinction between $Y$ and $Y'$ is captured by requiring the resolution $\tau \colon Y' \to Y$ to be $\CH_0$-trivial, i.e.\ $\tau_\ast$ induces isomorphisms on Chow groups of zero-cycles after any field extension.  
As an application, they showed that very general quartic threefolds do not admit a decomposition of the diagonal by specializing to a singular model of the Artin--Mumford example with a $\CH_0$-trivial resolution.  
Totaro \cite{totaro} further generalized the result, proving that very general hypersurfaces $X \subset \CP^{n+1}_\C$ of degree $d \geq 2\lceil (n+2)/3\rceil$ do not admit a decomposition of the diagonal. 
He specializes to mildly singular hypersurfaces in characteristic $2$ admitting regular differential forms by work of Koll\'ar \cite{kollar}.  

Hassett--Pirutka--Tschinkel \cite{HPT} used the techniques from \cite{voisin,CTP,CTO} to construct a smooth projective family $\mathcal X \to B$ of complex fourfolds with the remarkable property that rational fibers are dense (in the analytic topology), while a very general fiber is not (retract) rational. 
This settled the longstanding problem of whether rationality is an open condition in smooth families. 
The examples are quadric surface bundles over surfaces. 
Density of rational fibers follows from a Hodge-theoretic criterion of Voisin, while non-rationality of the very general fiber comes from a degeneration to a singular quadric surface bundle over $\CP^2$ with nontrivial Artin--Mumford invariant discovered in \cite{HPT}, see also \cite{pirutka}.

\subsection{Hypersurfaces under a logarithmic degree bound.}

In \cite{Sch-duke,Sch-JAMS}, the degeneration techniques of \cite{voisin,CTP} were refined so that an explicit resolution of singularities of the special fiber $Y$ and a detailed analysis of $Y' \to Y$ are often unnecessary.  
Instead, a more flexible cohomological vanishing criterion can be applied, see \cite[Proposition 3.1]{Sch-JAMS}.  
We state the result here in the context of unramified cohomology; see \cite{Sch-survey} for further details.

\begin{theorem}[\cite{Sch-duke,Sch-JAMS}] \label{thm:degeneration-O+V}
Let $R$ be a discrete valuation ring with algebraically closed residue field $k$ and fraction field $K$.
Let $\mathcal X\to \Spec R$ be a proper flat $R$-scheme with special fiber $Y=\mathcal X\times_R k$ and geometric generic fiber $X=\mathcal X\times_R \overline K$.
Assume that $X$ and $Y$ are irreducible and that the following holds:
\begin{enumerate}
    \item[(O)] There is a nontrivial class $\alpha\in H^i_{nr}(k(Y)/k,\mu_m^{\otimes j})$ for some integers $i$, $j$, and $m$, with $i\geq 1$;\label{item:thm:deg-Obstruction}
    \item[(V)] There is a regular alteration (e.g.\ a resolution) $\tau\colon Y'\to Y$ such that $\tau^\ast \alpha$ vanishes when restricted to any scheme point $x\in Y'$ with $\tau(x)\in Y^{\sing}$: $(\tau^\ast \alpha)|_x=0\in H^i(\kappa(x),\mu_m^{\otimes j})$.\label{item:thm:deg-Vanishing}
\end{enumerate}
Then $X$ does not admit a decomposition of the diagonal.
\end{theorem}

A smooth projective variety over an algebraically closed field has trivial unramified cohomology in positive degree: one sees this by acting with the diagonal on the respective unramified cohomology group and using the decomposition \eqref{eq:diagonal-decomposition}, see \cite[Theorem 7.4]{Sch-survey}.  
Condition (O) therefore implies that no resolution $Y'$ of $Y$ admits a decomposition of the diagonal.  
The vanishing condition (V) replaces the universal $\CH_0$-triviality requirement for a resolution.
The key point is that condition (V) is typically automatic for the unramified cohomology classes constructed via \cite{AM,CTO,pirutka,HPT,Sch-duke,Sch-JAMS}, see \cite[Theorem 9.2]{Sch-JAMS}.  
In practice, this often eliminates the need to explicitly compute a resolution or alteration of $Y$ (see e.g. \cite[Proposition 5.1]{Sch-JAMS}), which makes the method more flexible, especially in higher dimensions.

We then have the following more precise version of Theorem \ref{thm:JAMS-intro}; see \cite{Sch-ANT} for extensions to ${\rm char}(k)=2$.

\begin{theorem}[\cite{Sch-JAMS}] \label{thm:JAMS-body} 
Let $N\geq 3$ be an integer and write $N=n+r$ with positive integers $n$ and $r$ such that $ r \le 2^n-2$.  
Then a very general hypersurface $X\subset \CP^{N+1}_k$ of degree $d\geq n+2$ over a field of characteristic different from $2$ does not admit a decomposition of the diagonal.
\end{theorem}

In order to prove the above theorem we will apply Theorem \ref{thm:degeneration-O+V}.
To this end, we need, for each $d\ge n+2$, a hypersurface $Y\subset \CP^{N+1}_k$ of degree $d$ satisfying conditions (O) and (V) from Theorem \ref{thm:degeneration-O+V}.  
We describe the construction of $Y$, see \cite{Sch-JAMS} and \cite[Section 6]{Lange-Sch}. 
For simplicity, we focus on the extremal case $d=n+2$.  

Let $x_0,\dots,x_n,y_1,\dots,y_{r+1}$ be the coordinates of $\CP^{N+1}$ and let $t \in k$ be transcendental over its prime field.  
Consider the homogeneous polynomial
\begin{equation}\label{eq:gfromSch-torsion}
g(x_0,\dots,x_n) \coloneq t \left(\sum_{i=0}^n x_i^{\lceil (n+1)/2\rceil}\right)^2 - (-1)^n x_0^{2\lceil(n+1)/2\rceil - n} x_1 \cdots x_n \in k[x_0,\dots,x_n]
\end{equation}
of degree $\deg g = 2 \lceil (n+1)/2\rceil \le n+2$, and define
\begin{equation} \label{eq:F-JAMS}
F \coloneqq g(x_0,\dots,x_n) x_0^{2+n-\deg g} + \sum_{j=1}^{r} x_0^{n-\deg c_j} c_j(x_1,\dots,x_n) y_j^2 + (-1)^n x_1 \cdots x_n y_{r+1}^2,
\end{equation}
where 
\begin{equation}\label{eq:FermatPfister}
c_j(x_1,\dots,x_n) \coloneq (-x_1)^{\varepsilon_1} \cdots (-x_n)^{\varepsilon_n}, \quad j = \sum_i \varepsilon_i 2^{i-1}, \; \varepsilon_i \in \{0,1\}.
\end{equation}
The associated hypersurface is
\begin{equation} \label{eq:Z-Sch-torsion}
Y \coloneqq \{F = 0\} \subset \CP^{N+1}_k.
\end{equation}

Note that $Y$ has multiplicity $n=d-2$ along the $r$-plane $P=\{x_0=\dots=x_n=0\}$.  
Projection from $P$ gives a morphism $f\colon Bl_P Y \to \CP^n$, whose generic fiber $Q \subset \CP^{r+1}_{k(\CP^n)}$ is a quadric defined by the diagonal quadratic form
\[
q \coloneq \langle \tilde g, c_1, \dots, c_r, (-1)^n x_1 \cdots x_n \rangle
\]
over $k(x_1,\dots,x_n)$, where $\tilde g = g(1,x_1,\dots,x_n)$.  
This fibration into quadrics, and its close relation to the Pfister quadric  $\langle\langle x_1,\dots,x_n \rangle\rangle = \sum_{j=0}^{2^n-1} c_j y_j^2$, plays a major role in the argument.
The next theorem summarizes the key properties from \cite[Propositions 5.1, 6.1, 7.1]{Sch-JAMS}; for more details and some motivation how to find the above equations, see e.g.\ the survey \cite[Section 9]{Sch-survey}.

\begin{theorem} \label{thm:JAMS-unramified-class}
The class $f^\ast \alpha \in H^n(k(Y),\mu_2^{\otimes n})$ is unramified and nontrivial.  
Moreover, for any generically finite proper morphism $\tau\colon Y'\to Y$ and for any scheme-point $x\in Y'$ in the smooth locus, we have $(f^\ast \alpha)|_x = 0 \in H^n(\kappa(x),\mu_2^{\otimes n})$ if $\tau(x) \in P$ or if $\tau(x)\notin P$ and $f(\tau(x))$ is not the generic point of $\CP^n$.
\end{theorem}

Note the generic fiber of $f$ is smooth.
Therefore, $\tau(x) \in Y^{\sing}$ implies either $\tau(x) \in P$, or $\tau(x)\notin P$ and $f(\tau(x))$ is not the generic point of $\CP^n$.  
Theorem \ref{thm:JAMS-body} therefore follows from Theorems \ref{thm:degeneration-O+V} and \ref{thm:JAMS-unramified-class}.

\section{Motivic methods and degeneration to unions} \label{sec:motivic}

In \cite{NS,KT}, Nicaise--Shinder and Kontsevich--Tschinkel 
use the weak factorization theorem in characteristic zero to study a motivic birational invariant, namely the motivic volume, that is associated to a semi-stable degeneration of varieties in characteristic zero, see also \cite{NO2}. 
They use this to prove that (stable) rationality is a closed property in smooth projective families.
Another direct consequence of their work is the following obstruction for (stable) rationality in characteristic zero:

 \begin{theorem}[\cite{NS,KT}] \label{thm:NS-KT}
 Let $R$ be a discrete valuation ring with algebraically closed residue field $k$ of characteristic zero.
 Let $\pi\colon \mathcal X\to \Spec R$ be a flat proper $R$-scheme such that $\mathcal X$ is regular and the special fiber $Y=\mathcal X\times_R k$ is a simple normal crossing divisor on $\mathcal X$. 
 Let $X$ denote the geometric generic fiber of $\pi$.
Let $Y_i$ with $i\in I$ denote the irreducible components of $Y$.
Assume that the following holds in the free abelian group on (stably) birational equivalence classes of smooth $k$-varieties 
\begin{align} \label{eq:thm:NS-KT}
\sum_{\emptyset\neq J\subset I} (-1)^{|J|-1} [Y_J\times \CP_k^{|J|-1}] \neq [\CP_k^{\dim X}] \quad \quad \text{where}\quad \quad Y_J\coloneq \bigcap_{j\in J} Y_j.
\end{align} 
Then $X$ is not (stably) rational. 
 \end{theorem}

Recently, Hotchkiss and Stapleton upgraded the motivic volume of Nicaise--Shinder and Kontsevich--Tschinkel (i.e.\ 
the left hand side in \eqref{eq:thm:NS-KT}) to an invariant that is sensitive to the $R$-equivalence class of varieties and hence to retract rationality.
This leads in particular to a version of the above theorem, where the free abelian group of (stable) birational equivalence classes is replaced by the free abelian group on $R$-equivalence classes, see \cite[Corollary 1.5]{hotchkiss-stapleton}.
As in \cite{NS,KT}, their method uses the weak factorization theorem and hence resolution of singularities, which is the reason why the results are currently restricted to the case of characteristic zero.

In \cite{NO}, Nicaise and Ottem used Theorem \ref{thm:NS-KT} to prove the following, which was the first result on stable rationality beyond the logarithmic bound in Theorem \ref{thm:JAMS-body}.

\begin{theorem}[\cite{NO}] \label{thm:NO}
Let $X\subset \CP^6_\C$ be a very general quartic fivefold.
Then $X$ is not stably rational.
\end{theorem}
\begin{proof}[Sketch of proof]
Since stable rationality specializes in smooth projective families in characteristic zero by \cite{NS},  it suffices to exhibit one example of a smooth projective quartic fivefold over an algebraically closed field $k$ of characteristic zero that is not stably rational.
The idea of Nicaise and Ottem is to find a semi-stable degeneration $\mathcal X\to \Spec R$ over $R=\C[[t]]$, whose generic fiber is a quartic hypersurface and whose special fiber $Y=Y_0\cup Y_1$ decomposes into two isomorphic components $Y_0\cong Y_1$, such that the intersection $Z\coloneqq Y_0\cap Y_1$ is not stably rational by \cite{HPT}.
Explicit equations for $\mathcal X$ are given in \cite[Example 4.3.2]{NO2}; a more conceptual approach is taken in \cite{NO}.
In the free abelian group on stable birational equivalence classes, we then have
$$
[Y_0]+[Y_1]-[Z\times \CP^1]=2[Y_0]-[Z\times \CP^1]\neq [\CP^5_k] ,
$$
as can be seen after reduction modulo $2$.
The result thus follows from Theorem \ref{thm:NS-KT}.
\end{proof}

This method was taken further in the thesis of Moe (advised by Nicaise and Ottem), where he uses methods from discrete geometry to study rather involved degenerations of hypersurfaces into many components.
As a result, Moe shows that very general hypersurfaces $X\subset \CP_{\C}^{N+1}$ of degree $d\geq 4$ and dimension $N\leq (d+1)2^{d-4}$ are not stably rational. 
In joint work with Lange \cite{Lange-Sch}, which builds on previous joint work with Pavic \cite{Pavic-Sch}, this has been upgraded to the following cycle-theoretic result which implies Theorem \ref{thm:Lange-Sch-intro} from the introduction. 

\begin{theorem}[\cite{Lange-Sch}]\label{thm:Lange-Sch-main}
A very general hypersurface $X\subset \CP^{N+1}_k$ of degree $d\geq 4$ and dimension $N \leq (d+1)2^{d-4}$ over a field of characteristic different from $2$ does not admit a decomposition of the diagonal. 
\end{theorem}

As in \cite{NO,moe}, the strategy is to degenerate $X$ to a union of varieties.
In \cite{moe}, this union is combinatorially complicated, may contain many components, and the total space has toroidal singularities.
In \cite{Lange-Sch}, instead, the crucial idea is to work with non-proper but semi-stable degenerations, thus removing all singularities and all but two components of the special fiber, whose intersection is ``sufficiently irrational''. 

To formulate the obstruction, we say that a $k$-variety $X$ admits a $\Lambda$-decomposition of the diagonal for some ring $\Lambda$ if there is a zero-cycle $z\in \CH_0(X)$ with $\delta_X=z_{k(X)}\in \CH_0(X_{k(X)})\otimes_\Z \Lambda$, cf.\ \eqref{eq:diagonal-decomposition-zero-cycle}.
If $X$ admits a $\Lambda$-decomposition of the diagonal, then so does any open subset of $X$, cf.\ \cite[Proposition 1.8]{fulton}.
Moreover, by Fulton's specialization map on Chow groups, if a variety $X$ admits a $\Lambda$-decomposition of the diagonal then so does any specialization $Y$ of $X$, cf.\ \cite[\S 20.3]{fulton} and \cite[Lemma 3.8]{Lange-Sch}. 
Let us finally recall that the exponential characteristic $e$ of a field $k$ is $1$ if ${\rm char}(k)=0$ and $e={\rm char}(k)$ otherwise. 

\begin{theorem}[\cite{Lange-Sch}]\label{thm:Obstruction-Lange-Sch}
 Let $R$ be a discrete valuation ring with algebraically closed residue field $k$ and fraction field $K$. 
Let $\Lambda$ be a torsion ring in which the exponential characteristic of $k$ is invertible.
 Let $\mathcal{X} \to \Spec R$ be a flat separated (not necessarily proper) $R$-scheme of finite type with smooth irreducible geometric generic fiber $X=\mathcal X\times_R\overline K$ and special fiber $Y=\mathcal X\times_Rk$. 
    Assume that 
    \begin{enumerate}[label={(\arabic*)}]
        \item 
        $Y$ is a simple normal crossing divisor on $\mathcal X$ (hence $\mathcal X$ is strictly semi-stable, cf.\ \cite[Definition 2.2]{Lange-Sch}); 
        \item $Y=Y_0\cup Y_1$ consists of two components that are both isomorphic to open subsets of affine space: $Y_i\subset \A^{\dim Y_i}$. 
    \end{enumerate} 
   If $X$ admits a $\Lambda$-decomposition of the diagonal, then each component of $Z\coloneq Y_0\cap Y_1$ admits a $\Lambda$-decomposition of the diagonal. 
\end{theorem}
\begin{proof}
Assume that $X$ admits a $\Lambda$-decomposition of the diagonal.
Then there is a closed zero-dimensional subset $W_X\subset X$ such that $\delta_X$ vanishes in $\CH_0(U\times_{\bar K}\bar K(X))\otimes_\Z \Lambda$, where $U=X\setminus W_X$.
Up to enlarging $W_X$, we can assume that it is defined over $K$ and hence we can take its closure  $W_{\mathcal X}\subset \mathcal X$.
Let $\mathcal X^\circ\coloneq \mathcal X\setminus W_{\mathcal X}$ with special fiber $Y^\circ=Y_0^\circ\cup Y_1^{\circ}$ and $Z^\circ=Y_0^\circ\cap Y_1^{\circ}=Z\setminus W_Z$, 
where $W_Z\coloneq Z\cap W_{\mathcal X}$.
For any field extension $L/k$, the map 
    $$
        \Psi_{Y^\circ_L} \colon (\CH_1(Y^\circ_0 \times_k L)\otimes_\Z \Lambda) \oplus (\CH_1(Y^\circ_1 \times_k L)\otimes_\Z \Lambda) \longrightarrow \CH_0(Z^\circ \times_k L)\otimes_\Z \Lambda,\quad (\gamma_0,\gamma_1) \mapsto  
        \gamma_0|_{Z^\circ}-\gamma_1|_{Z^\circ} 
    $$
is surjective by \cite[Theorem 4.2]{Lange-Sch}.
Since $Y_i$ and hence $Y_i^\circ$ is an open subset of affine space, the source of the map $\Psi_{Y^\circ_L}$ is trivial.
Hence, $\CH_0(Z^\circ \times_k L)\otimes_\Z \Lambda=0$ for all $L/k$.
This implies that each component of $Z$ admits a $\Lambda$-decomposition of the diagonal.
To see this, we may assume that $Z$ is irreducible. 
The localization sequence \cite[Proposition 1.8]{fulton} then implies that the diagonal point $\delta_Z\in \CH_0(Z_{k(Z)})\otimes_\Z \Lambda$ lies in the image of $\CH_0((W_Z)_{k(Z)})\otimes_\Z \Lambda$, which proves the claim, because $W_Z$ is a finite set of points. 
This concludes the proof.
\end{proof}

Note that a variety which admits a decomposition of the diagonal also admits a $\Lambda$-decomposition of the diagonal for any ring $\Lambda$.
Since $\dim Z=\dim X-1$, the above theorem therefore yields an inductive approach to the non-existence of ($\Lambda$-)decompositions of the diagonal.
In \cite{Lange-Sch}, this is applied to degenerations that are motivated by Moe's double cone construction in \cite{moe}, which we explain next.

\begin{example}[Double cone construction \cite{moe,Lange-Sch}] \label{ex:double-cone}
Let $d\geq 4$ be an integer and let 
$$
g,h_1,h_2\in   k[x_0,x_1,\dots ,x_n,y_2,y_3,\dots ,y_{r+1}]
$$ 
be homogeneous polynomials of degrees $|g|=d$, $|h_i|=d-2i$ and consider
$$
f\coloneqq g+h_1x_0y_1 +h_2x_0^2y_1^2 + x_0^{d-1} z + x_0^{d-2} y_1 w\in k[x_0,x_1,\dots ,x_n,y_1,y_2,y_3,\dots ,y_{r+1},z,w]
$$
Note that $g$ and $h_1,h_2$ do not contain the variable $y_1$.
Consider the discrete valuation ring $R=k[[t]]$ and the flat projective $R$-scheme given by
\begin{align} \label{eq:mathcalX-doublecone}
\mathcal{X} \coloneqq \{t x_0^2 + zw = f = 0\} \subset \CP^{N+3}_R ,
\end{align}
where $N\coloneq n+r$.
We aim to describe the geometric generic fiber $X$ and the special fiber $Y$, respectively:
\begin{itemize}
\item to describe the geometric generic fiber $X$, we eliminate $w$ via $w=-tx_0^2/z$ and make the substitution $y_1\to zy_1/x_0$.
We then find that  $X$ is birational to the degree $d$ hypersurface given by 
\begin{align} \label{eq:double-cone-generic-fiber}
\{g+h_1zy_1 +h_2z^2y_1^2 + x_0^{d-1} z -tx_0^{d-1}y_1 =0\} \subset \CP^{N+2}_{\Frac R}.
\end{align}
    \item the special fiber $Y$ decomposes into a union $Y = Y_0 \cup Y_1$, where
    $$
    Y_0=\{g+h_1x_0y_1+h_2x_0^2y_1^2+x_0^{d-1}z=0\}\subset \CP^{N+2}_k\quad \text{and}\quad Y_1=\{g+h_1x_0y_1+h_2x_0^2y_1^2+x_0^{d-2}y_1w=0\}\subset \CP^{N+2}_k .
    $$
    These are hypersurfaces of degree $d$ that have a point of multiplicity $d-1$ (given by the vanishing of all $x$ and $y$-variables), hence these hypersurfaces are rational (as long as they are irreducible).
    Their intersection is given by the degree $d$ hypersurface
$$
    Z \coloneqq  \{g+h_1x_0y_1+h_2x_0^2y_1^2=0\}\subset \CP^{N+1}_k .
$$
\end{itemize}
\end{example}

We are now in the position to sketch the proof of Theorem \ref{thm:Lange-Sch-main}.

\begin{proof}[Sketch of proof of Theorem \ref{thm:Lange-Sch-main}]
Let $d\geq 4$ and $N\leq (d+1)2^{d-4}$. 
We fix $d$ and aim to apply induction on $N$.
If $N=n+r$, $n=d-2$ with $r\leq 2^{n}-2$, then the result follows from Theorem \ref{thm:JAMS-body}. 
This is the base case of the induction. 
Let us explain the next step, i.e.\ how to go from $N=n+r$, $n=d-2$ and $r=2^{n}-2$ to $N+1$. 
We apply the double cone construction as in Example \ref{ex:double-cone} such that the hypersurface $Z\subset \CP^{N+1}_k$ is cut out by equation \eqref{eq:F-JAMS}.
We then replace the given degeneration $\mathcal X\to \Spec R$ from \eqref{eq:mathcalX-doublecone} by an open subset $\mathcal X^\circ$ with special fiber $Y^\circ=Y_1^\circ\cup Y_2^\circ$ and $Z^\circ\coloneq Y_1^\circ\cap Y_2^\circ$.
We choose this open subset such that $\mathcal X^\circ$ is regular and semi-stable over $R$, and such that $Y_i^\circ\subset \mathbb A^{\dim Y_i}$.
This amounts to removing several closed subsets from $\mathcal X$; we choose them carefully so that the arguments in \cite{Sch-JAMS} still imply that the (non-proper) variety $Z^\circ$ does not admit a $\Z/2$-decomposition of the diagonal.
(In fact, to make this work one has to change the degeneration slightly, and replace the monomial $x_0^{d-2}y_1w$ in the definition of $f$ by $x_0^{d-2}(x_0+\lambda y_1)w $ for some parameter $\lambda$, but we ignore these additional difficulties here, cf.\ \cite[(5.3)]{Lange-Sch}.)
Applying Theorem \ref{thm:Obstruction-Lange-Sch} to this set-up, we conclude that the geometric generic fiber $X^\circ$, which is an open subset of the degree $d$ hypersurface \eqref{eq:double-cone-generic-fiber} of dimension $N+1$, does not admit a $\Z/2$-decomposition of the diagonal.
We may arrange that $X^\circ$ is smooth and so it follows by applying Fulton's specialization map on Chow groups  \cite[\S 20.3]{fulton} that the very general hypersurface of degree $d$ and dimension $N+1$ does not admit a $\Z/2$-decomposition of the diagonal.
We have thus explained how to pass from degree $d$ and dimension $N$ (with $N=n+2^{n}-2$ and $d=n+2$) to degree $d$ and dimension $N+1$.

To iterate the procedure, one now replaces the role of the equation \eqref{eq:F-JAMS} from \cite{Sch-JAMS} by the equation of the generic fiber \eqref{eq:double-cone-generic-fiber} of the double cone construction from the previous step.
This works as long as $h_i$ has $x_0^i$ as a factor.  
After at most $(d-2)/2$ steps this fails, but we can relabel the $y_j$ and continue whenever some $y_j$ appears only linearly or quadratically in monomials divisible by sufficiently high powers of $x_0$.
In \eqref{eq:F-JAMS} each $y_j$ appears only in the term $x_0^{d-\deg c_j-2}c_jy_j^2$, where $c_j=\pm \prod x_i^{\epsilon_i}$ is a square-free monomial in $x_1,\dots,x_n$ (cf.\ \eqref{eq:FermatPfister}).
For $j=1,\dots,2^n-1$, the $c_j$ essentially run through all such monomials, whose average degree is about $n/2=(d-2)/2$.
Thus the process can be repeated roughly $r\cdot \lfloor (d-2)/4\rfloor$ times. 
This covers all cases
\[
N\leq d-2+ 2^{d-2}+(2^{d-2}-2)\cdot \lfloor (d-2)/4\rfloor  \approx  2^{d-4}(d+1),
\]
which explains the order of magnitude of the bound in Theorem \ref{thm:Lange-Sch-main}.

We stress that this outline only conveys the main idea; the full proof requires resolving several technical difficulties, for more details see \cite{Lange-Sch}. 
\end{proof}

\section{Obstructing cohomological decompositions of the diagonal of cubic threefolds} \label{sec:coho-decomposition}

Let $X$ be a smooth complex projective threefold that is rationally connected.
Then $h^{3,0}(X)=0$ and the intermediate Jacobian 
$$
JX=H^{1,2}(X)/H^3(X,\Z) ,
$$
introduced by Weil and Griffiths,
is a principally polarized abelian variety, whose theta divisor $\Theta_X\subset JX$, induced by cup product and Poincar\'e duality, is unique up to translation.

\begin{theorem}[\cite{clemens-griffiths}]
Let $X$ be a smooth projective rationally connected threefold over $\C$.
 If $X$ is rational, then $(JX,\Theta_X)$ is isomorphic to a product of Jacobians of curves.
The latter fails in the case of a smooth cubic threefold, which is therefore not rational.
\end{theorem}

By Matsusaka's criterion, $(JX,\Theta_X)$ is isomorphic to a product of Jacobians of curves if and only if the minimal cohomology class
\begin{align} \label{eq:min-class-JX}
[\Theta_X]^{g-1}/(g-1)!\in H_2(JX,\Z)
\end{align}
is represented by an effective curve, where $g\coloneq \dim JX=b_3(X)/2$.
In \cite{voisin-JEMS}, Voisin generalized this criterion to stable (or retract) rationality by asking that \eqref{eq:min-class-JX} is represented by some (not necessarily effective) linear combination of curves, see Theorem \ref{thm:coho-dec-diagonal-complex-threefolds} below. 
This obstruction has recently been applied successfully in \cite{EGFS}.
In the following two subsections, we explain the parts of Voisin's work \cite{voisin-JEMS,voisin-min-class} that are used in \cite{EGFS} in some detail.
We start for convenience with a more general set-up over arbitrary fields in Section \ref{subsec:coho-dec-diag-1} and specialize to complex rationally connected varieties only in Section \ref{subsec:coho-dec-diag-2}.

\subsection{Cohomological decompositions of the diagonal.} \label{subsec:coho-dec-diag-1}
We follow Voisin’s ideas in \cite[Theorem 3.1]{voisin-JEMS} and describe them in slightly greater generality. 

Let $k$ be an arbitrary field.
For concreteness, we fix a prime $\ell$ invertible in $k$ and for an algebraic $k$-scheme $X$ we consider the twisted cohomology theory
$$
H^i(X,n)\coloneq\begin{cases}
    H^i_{cont}(X_{\text{\'et}},\Z_\ell(n))=H^i(X_{\text{pro\'et}},\Z_\ell(n))\quad &\text{if $k\neq \C$};\\
    H^i(X,\Z)\coloneq H^i_{sing}(X(\C),\Z)\quad &\text{if $k=\C$}.
\end{cases}
$$
In the above situation, Jannsen's continuous \'etale cohomology \cite{jannsen} agrees with $\ell$-adic pro-\'etale cohomology of Bhatt--Scholze \cite{bhatt-scholze}.
For basic properties of this cohomology theory that we will use, see \cite[Appendix A]{Sch-moving}.

\begin{remark}
The results presented in this subsection work more generally for any twisted cohomology theory with support which satisfies some basic properties, such as the  Bloch--Ogus axioms in \cite[\S 1]{BO}.
In fact, we only need: cup products, pullbacks, cycle classes, proper pushforwards along proper morphisms of smooth projective equi-dimensional schemes, and the projection formula, together with some natural compatibilities.
\end{remark} 

Furthermore, let $X,Y$ be smooth projective equi-dimensional $k$-schemes of dimensions $d_Y\coloneq \dim Y$ and  $d_X\coloneq \dim X$.
For some $n\geq 0$, let $\Gamma\in \CH_{n}(Y\times X)$ be an $n$-dimensional cycle on $Y\times X$.
Then, as usual, we define an action
$$
\Gamma^\ast\colon H^i(X,m)\longrightarrow H^{i+2\dim Y-2n}(Y,m+\dim Y-n),\quad \quad \alpha\mapsto  p_\ast( \cl(\Gamma)\cup q^\ast\alpha) ,
$$
where $p\colon Y\times X\to Y$ and $q\colon Y\times X\to X$ denote the natural projections.
If $\Gamma=\Gamma_f$ is the graph of a morphism $f\colon Y\to X$, then $\Gamma^\ast \alpha=f^\ast \alpha$.

\begin{definition}[\cite{voisin-JEMS,voisin-survey}] \label{def:coho-decomposition-of-diagonal}
Let $X$ be a smooth projective $k$-variety of dimension $n$.
We say that $X$ admits a \emph{cohomological decomposition of the diagonal} (with respect to the chosen cohomology theory) if  there is a cycle $\Gamma\in \CH_n(X\times X)$ whose support is contained in $E\times X$ for some divisor $E\subset X$ and a zero-cycle $z\in \CH_0(X)$, such that
\begin{align} \label{eq:def:coho-decomposition-of-diagonal}
\cl(\Delta_X-X\times z-\Gamma)=0\in H^{2n}(X\times X,n) .
\end{align}
\end{definition}

The existence of cycle class maps ensures that a variety which admits a decomposition of the diagonal (see Definition \ref{def:coho-decomposition-of-diagonal}) also admits a cohomological decomposition.

\begin{lemma} \label{lem:Gamma*cup-1}
Let $X,Y$ be smooth projective $k$-varieties with $n\coloneq \dim Y$.
Let $D\subset Y$ be a smooth irreducible divisor.
Denote the natural inclusion by $\iota \colon D\to Y$ and let $\Gamma\in \CH_n(D\times X)$ be a cycle on $D\times X$.
Let $Z_1,Z_2\subset D$ be smooth divisors such that $\mathcal O_X(D)|_D=\mathcal O_D(Z_1-Z_2)$ and let $f_i\colon Z_i\hookrightarrow Y$ and $g_i\colon Z_i\hookrightarrow D$ denote the natural inclusions.
Let further $\Gamma_i\in \CH_{n-1}(Z_i\times X)$ denote the pullback of $\Gamma$ along $g_i\times \id \colon Z_i\times X\hookrightarrow D\times X$.
Then the following holds for all $\alpha\in H^j(X,a)$ and all $\beta\in H^l(X,b)$: 
$$
\iota_{\ast}\Gamma^\ast \alpha\cup \iota_{\ast}\Gamma^\ast \beta=f_{1\ast} ( \Gamma_1^\ast\alpha\cup \Gamma_1^\ast\beta )-f_{2\ast} ( \Gamma_2^\ast\alpha\cup \Gamma_2^\ast\beta )  .
$$ 
\end{lemma}
\begin{proof}
By the projection formula, we have
\begin{align*}
\iota_{\ast}\Gamma^\ast \alpha\cup \iota_{\ast}\Gamma^\ast \beta=\iota_{\ast}(\iota^\ast\iota_\ast\Gamma^\ast \alpha\cup \Gamma^\ast \beta)
&=\iota_{\ast}(c_1(\mathcal O_X(D)|_D)\cup \Gamma^\ast \alpha\cup \Gamma^\ast \beta) \\
&=\iota_{\ast}(\cl_D(Z_1)\cup \Gamma^\ast \alpha\cup \Gamma^\ast \beta)-\iota_{\ast}(\cl_D(Z_2)\cup \Gamma^\ast \alpha\cup \Gamma^\ast \beta) .
\end{align*}
Note that  $\cl_D(Z_i)=g_{i\ast}\cl_{Z_i}(Z_i)=g_{i\ast}1$.
Moreover, since $f_i=\iota\circ g_i $, the projection formula yields
$$ 
\iota_{\ast}((g_{i\ast} 1\cup \Gamma^\ast \alpha\cup \Gamma^\ast \beta) 
=f_{i\ast} (1\cup g_i^\ast \Gamma^\ast\alpha\cup g_i^\ast \Gamma^\ast\beta )=f_{i\ast} (g_i^\ast \Gamma^\ast\alpha\cup g_i^\ast \Gamma^\ast\beta ).
$$
The result follows now from the fact that $g_i^\ast\circ \Gamma^\ast=\Gamma_i^\ast$.
\end{proof}

\begin{lemma} \label{lem:Gamma*cup-2}
Let $X,Y$ be smooth projective $k$-varieties with $n\coloneq \dim Y$.
Let $D_1,D_2\subset Y$ be smooth irreducible divisors with $D_1\neq D_2$ and  smooth intersection $Z\coloneq D_1\cap D_2$.
Denote the natural inclusions by $\iota_1\colon D_1\to Y$, $\iota_2\colon D_2\to Y$, $g_1\colon Z\to D_1$, $g_2\colon Z\to D_2$ and $f\colon Z\to Y$.
For $i=1,2$, let $\Gamma_i\in \CH_n(D_i\times X)$ be a cycle on $D_i\times X$.
Let further $\Gamma_{ii}\in \CH_{n-1}(Z\times X)$ be the pullback of $\Gamma_i$ along $g_i\times \id$.
Then the following holds for all $\alpha\in H^j(X,a)$ and all $ \beta\in H^l(X,b)$:
\begin{align*}
 \iota_{1\ast}\Gamma_1^\ast \alpha\cup \iota_{2\ast}\Gamma_2^\ast \beta =f_\ast( \Gamma_{11}^\ast \alpha \cup \Gamma_{22}^\ast \beta) 
 \end{align*}
 and
 \begin{align*}
 (\iota_{1\ast}\Gamma_1^\ast \alpha \cup \iota_{2\ast}\Gamma_2^\ast \beta)+(\iota_{2\ast}\Gamma_2^\ast \alpha\cup \iota_{1\ast}\Gamma_1^\ast \beta)&=\\
    f_\ast( (\Gamma_{11}+\Gamma_{22})^\ast \alpha &\cup (\Gamma_{11}+\Gamma_{22})^\ast \beta) -f_\ast( \Gamma_{11}^\ast \alpha \cup \Gamma_{11}^\ast \beta)-f_\ast( \Gamma_{22}^\ast \alpha \cup \Gamma_{22}^\ast \beta) .
    \end{align*} 
\end{lemma}
\begin{proof}
The second identity is a formal consequence of the first.
To prove the first identity, we follow the same argument as in Lemma \ref{lem:Gamma*cup-1}.
By the projection formula, we have
$$
\iota_{1\ast}\Gamma_1^\ast \alpha\cup \iota_{2\ast}\Gamma_2^\ast \beta=
 \iota_{2\ast}( (\iota_{2}^\ast\iota_{1\ast}\Gamma_1^\ast \alpha) \cup\Gamma_2^\ast \beta) .
$$
On the level of correspondences, one checks that $\iota_{2}^\ast\iota_{1\ast}=g_{2\ast}g_1^\ast$.
The above identity thus simplifies to
$$
\iota_{1\ast}\Gamma_1^\ast \alpha\cup \iota_{2\ast}\Gamma_2^\ast \beta=
 \iota_{2\ast}( (g_{2\ast}g_{1}^\ast \Gamma_1^\ast \alpha) \cup\Gamma_2^\ast \beta) =
  \iota_{2\ast}( (g_{2\ast}\Gamma_{11}^\ast \alpha) \cup\Gamma_2^\ast \beta).
$$
Using the projection formula once again, we thus get
$$
\iota_{1\ast}\Gamma_1^\ast \alpha\cup \iota_{2\ast}\Gamma_2^\ast \beta=
  \iota_{2\ast} g_{2\ast} ( \Gamma_{11}^\ast \alpha \cup g_2^\ast \Gamma_2^\ast \beta) =f_{\ast}( \Gamma_{11}^\ast \alpha \cup \Gamma_{22}^\ast \beta) .
$$
This concludes the proof.
\end{proof}

\begin{proposition} \label{prop:Gamma*cup-1+2}
Let $X$ and $Y$ be smooth projective $k$-varieties, let $D\subset Y$ be a simple normal crossing divisor and let $\Gamma\in \CH_n(Y\times X)$ be a cycle whose support is contained in $D\times X$.
Then there are smooth projective equi-dimensional schemes $Z_1,Z_2$ of dimension $\dim Z_1=\dim Z_2=n-2$, morphisms $f_i\colon Z_i\to Y$ and cycles $\Gamma_i\in \CH_{n-1}(Z_i\times X)$ such that for all $\alpha\in H^j(X,a)$ and all $ \beta\in H^l(X,b)$:
$$
\Gamma^\ast \alpha\cup \Gamma^\ast \beta = f_{1\ast} ( \Gamma_1^\ast\alpha\cup \Gamma_1^\ast\beta )-f_{2\ast} ( \Gamma_2^\ast\alpha\cup \Gamma_2^\ast\beta ) \in H^{i+l}(Y,a+b) .
$$
\end{proposition}
\begin{proof}
Let $D_i$, $i\in I$, denote the components of $D$.
Let $\iota_i\colon D_i\to Y$ be the inclusion and write $\Gamma=\sum_i \iota_{i\ast}\Gamma_i$ with $\Gamma_i\in \CH_n(D_i\times X)$.
The result in the proposition is then a direct consequence of Lemmas \ref{lem:Gamma*cup-1} and \ref{lem:Gamma*cup-2}.
\end{proof}

For simplicity, we will now assume that $k$ is perfect.
For a  closed point $z\in X$ of an algebraic $k$-scheme $X$ and a class $\alpha\in H^i(X,n)$, we define $\langle z,\alpha\rangle\in H^i(\Spec k,n)$ as the pushforward of $\alpha|_z\in H^i(\Spec \kappa(z),n)$ via the structure map $\Spec \kappa(z)\to \Spec k$.
This extends to zero-cycles by linearity.
Note that this class is automatically zero if $H^i(\Spec k,n)=0$, which holds e.g.\ if $k$ is algebraically closed and $i>0$.

\begin{theorem} \label{thm:coho-dec-diagonal}
Let $X$ be a smooth projective variety of dimension $n$ over a perfect field $k$.
Assume that $X$ admits a cohomological decomposition of the diagonal. 
Then,  for $i=1,2$, there is a smooth projective equi-dimensional scheme $Z_i$ of dimension $n-2$, a morphism $f_i\colon Z_i\to X$ and a cycle $\Gamma_i\in \CH_{n-1}(Z_i\times X)$, such that 
for all $\alpha\in H^j(X,a)$ and all $ \beta\in H^l(X,b)$ with $\langle z,\alpha \rangle=0=\langle z,\beta\rangle $, where $z\in \CH_0(X)$ is as in \eqref{eq:def:coho-decomposition-of-diagonal},  we have
$$
e^m\cdot (\alpha\cup\beta) 
= f_{1\ast}( \Gamma_1^\ast\alpha\cup \Gamma_1^\ast\beta ) - f_{2\ast} (\Gamma_2^\ast\alpha\cup \Gamma_2^\ast\beta ) \in H^{j+l}(X,a+b) ,
$$
where $e^m$ denotes some power of the exponential characteristic of $k$.
\end{theorem}
\begin{proof}
By assumption, there is a divisor $E\subset X$ and a cycle $\Gamma\in \CH_n(X\times X)$ that is supported on $D\times X$, such that $\cl(\Delta_X-X\times z-\Gamma)=0$.
Let $\tau\colon Y\to X$ be a regular alteration, i.e.\ a generically finite proper morphism from a regular variety $Y$, such that $D=\tau^{-1}(E)$ is a simple normal crossing divisor on $Y$.
By work of Temkin \cite{temkin}, we may assume that $\deg \tau=e^m$ is a power of the exponential characteristic of $k$.
We then use Fulton's \cite{fulton} refined pullback map on cycles to arrive at an $n$-dimensional cycle $\Gamma'=(\tau\times \id)^\ast \Gamma$ on $Y\times X$ whose support is contained in $D\times X$.
Note that $(\tau\times \id)^\ast\Delta_X=\Gamma_{\tau}$ is the graph of $\tau$.
Since pullbacks are compatible with cycle classes, we thus find that
$$
\cl(\Gamma_\tau-(Y\times z)-\Gamma')=0\in H^{2n}(Y\times X,n) .
$$ 
By assumptions,  $\langle z,\alpha \rangle=0$ and $\langle z,\beta\rangle =0$ and so $[Y\times z]^\ast \alpha$ and $[Y\times z]^\ast \beta$ vanish.
Hence,  
$$
\tau^\ast(\alpha\cup \beta) =\tau^\ast \alpha\cup \tau^\ast \beta =\Gamma_\tau^\ast \alpha\cup \Gamma_\tau^\ast \beta = {\Gamma'}^\ast\alpha\cup {\Gamma'}^\ast\beta \in H^{j+l}(Y,a+b)  .
$$ 
The result now follows from Proposition \ref{prop:Gamma*cup-1+2} by applying $\tau_\ast$ to both sides, because $\tau_\ast\circ \tau^\ast=\deg(\tau)\cdot \id$.  
\end{proof}

For convenience, let us state what Theorem \ref{thm:coho-dec-diagonal} says for $k=\C$; this recovers parts of \cite[Theorem 3.1]{voisin-JEMS}.

\begin{corollary} \label{cor:coho-dec-diagonal-complex}
Let $X$ be an $n$-dimensional smooth complex projective variety which admits a cohomological decomposition of the diagonal.
Then there are smooth projective equi-dimensional schemes $Z_1,Z_2$ of dimension $n-2$, morphisms $f_i\colon Z_i\to X$ and cycles $\Gamma_i\in \CH_{n-1}(Z_i\times X)$ for $i=1,2$, such that 
for all $\alpha\in H^j(X,\Z)$ and all $ \beta\in H^l(X,\Z)$ with $j,l\geq 1$,
we have
$$
\alpha\cup\beta = f_{1\ast}( \Gamma_1^\ast\alpha\cup \Gamma_1^\ast\beta ) - f_{2\ast} (\Gamma_2^\ast\alpha\cup \Gamma_2^\ast\beta ) \in H^{j+l}(X,\Z) .
$$ 
\end{corollary}

\begin{remark}
Assume for simplicity that $k$ has characteristic zero (so $e^m=1$). 
Theorem \ref{thm:coho-dec-diagonal} and Corollary \ref{cor:coho-dec-diagonal-complex} are most effective in degree $j=l=\dim X$. 
Indeed, if $X$ admits a cohomological decomposition of the diagonal, then, by Theorem \ref{thm:coho-dec-diagonal}, its middle cohomology embeds into the middle cohomology of a smooth equi-dimensional projective scheme $Z=Z_1\sqcup Z_2$ of dimension $\dim X-2$. 
This embedding is induced by an algebraic cycle and it is compatible with cup products up to a sign change on the pairing for $Z_2$.
For rational $X$, the stronger case $Z_2=\emptyset$ follows directly from the weak factorization theorem. 
For stably or retract rational $X$, the conclusion still holds because such $X$ admit a decomposition of the diagonal, but it is less clear how to obtain the result directly from blow-up formulas and weak factorization, without using algebraic cycles.
\end{remark}

\subsection{Voisin's generalization of the Clemens--Griffiths method.} \label{subsec:coho-dec-diag-2}

In the case of rationally connected threefolds, Corollary \ref{cor:coho-dec-diagonal-complex} implies the following, see \cite[Theorem 4.1]{voisin-JEMS} and \cite[Theorem 1.4]{voisin-min-class}. 

\begin{theorem}[Voisin \cite{voisin-JEMS,voisin-min-class}] \label{thm:coho-dec-diagonal-complex-threefolds}
Let $X$ be a rationally connected smooth complex projective threefold with intermediate Jacobian $JX$ of dimension $g$.
Assume that $X$ admits a cohomological decomposition of the diagonal.
Then there are smooth projective (possibly disconnected) curves $C_1,C_2$, and a morphism
$$
f\colon JC_1\times JC_2\longrightarrow JX  
$$ 
such that
\begin{enumerate}[label={(\arabic*)}]
    \item There is an abelian variety $B$ and an isomorphism $JC_1\times JC_2\cong JX\times B$ as unpolarized abelian varieties;  \label{item:thm:min-class:2}
    \item The minimal class $[\Theta_X]^{g-1}/(g-1)!$ is algebraic, namely given by $f_\ast [C_1\times 0]- f_\ast[0\times C_2]$, where $C_i\stackrel{AJ}\hookrightarrow JC_i$. \label{item:thm:min-class:1}  
\end{enumerate} 
\end{theorem}
\begin{proof}
Recall first the following fact: if $A$ is an abelian variety with principal polarization $[\Theta]\in H^2(A,\Z)$, then it induces an isomorphism between $A$ and its dual $A^\vee$ and we get canonical isomorphisms
\begin{align} \label{eq:H^i(A)=H_i(A)}
H^i(A,\Z)=H_i(A,\Z)^\vee = H_i(A^\vee,\Z)\cong H_i(A,\Z).
\end{align}
This applies in particular to $JX$ and to Jacobians of curves; it will be used frequently in what follows.

By Corollary \ref{cor:coho-dec-diagonal-complex},
there are smooth projective (possibly disconnected) curves $C_1,C_2$, morphisms $f_i\colon C_i\to JX$ and cycles $\Gamma_i\in \CH_{2}(C_i\times X)$ for $i=1,2$ such that for all $\alpha,\beta\in H^3(X,\Z)$ we have
\begin{align} \label{eq:cup-product}
\langle \alpha,\beta\rangle_X=\langle \Gamma_1^\ast \alpha,\Gamma_1^\ast\beta\rangle_{C_1}-\langle \Gamma_2^\ast\alpha,\Gamma_2^\ast\beta\rangle_{C_2} ,
\end{align}
where $\langle\, ,\, \rangle_X$ and $\langle\, ,\, \rangle_{C_i}$ denote the natural cup product pairings on $H^3(X,\Z)$ and $H^1(C_i,\Z)$, respectively.
Consider the morphism of integral Hodge structures
\begin{align} \label{eq:fast}
H^1(JX,\Z)=H^3(X,\Z) \longrightarrow 
H^1(C_1,\Z)\oplus H^1(C_2,\Z) ,\quad \alpha \mapsto (\Gamma_1^\ast \alpha,\Gamma_2^\ast \alpha) ,
\end{align}
and identify $\langle\, ,\, \rangle_X$ with the induced pairing on $H^1(JX,\Z)\cong H_1(JX,\Z)$.
By the correspondence between polarizable weight one Hodge structures and abelian varieties, \eqref{eq:fast} yields a unique morphism of abelian varieties 
$$
f\colon JC_1\times JC_2\longrightarrow JX \quad \quad \text{with dual}\quad \quad f^\vee\colon JX\longrightarrow JC_1\times JC_2 ,
$$
such that \eqref{eq:fast} is given by $f^\ast$.
Moreover, if we use \eqref{eq:H^i(A)=H_i(A)} to identify \eqref{eq:fast} with a map between homology, then it agrees with $f^\vee_\ast$. 
The symmetric bilinear pairing $\langle\, ,\, \rangle_{C_1}-\langle\, ,\, \rangle_{C_2}$ on
$$
H^1(C_1,\Z)\oplus H^1(C_2,\Z)=H^1(JC_1\times JC_2,\Z)\cong H_1(JC_1\times JC_2,\Z)= H_1(C_1,\Z)\oplus H_1(C_2,\Z)
$$ 
is unimodular, but not positive definite (unless $JC_2=0$); we refer to it as a quasi-polarization.
The compatibility in \eqref{eq:cup-product} says precisely that this quasi-polarization restricts to the pairing $\langle\, ,\, \rangle_X$ on 
$$
\im(f^\vee_\ast\colon H_1(JX,\Z)\longrightarrow H_1(JC_1,\Z)\oplus H_1(JC_2,\Z)).
$$
Since $\langle\, ,\, \rangle_X$ is positive definite, we may define  $B\subset JC_1\times JC_2$ as the abelian subvariety such that $H_1(B,\Z)=\im(f^\vee_\ast)^\perp$ is the orthogonal complement of $\im(f^\vee_\ast)$ with respect to the quasi-polarization  $\langle\, ,\, \rangle_{C_1}-\langle\, ,\, \rangle_{C_2}$.
(One can check that in fact $B$ is the connected component of $0$ of $\ker f$.)
The natural map $JX\times B\to JC_1\times JC_2$ then induces an isomorphism on $H_1$, hence is an isomorphism of unpolarized abelian varieties, as claimed in \ref{item:thm:min-class:2}.
 
It remains to prove item \ref{item:thm:min-class:1}.
As a consequence of \ref{item:thm:min-class:2}, we know that $f^\vee$ is an embedding: $X\cong \im(f^\vee)$.
The quasi-polarization $\langle \, ,\, \rangle_{C_1}-\langle\, ,\, \rangle_{C_2}$ on $H^1(JC_1\times JC_2,\Z)$ corresponds to the class 
$$
\pr_1^\ast[\Theta_{C_1}]-\pr_2^\ast[\Theta_{C_2}]\in \Lambda^2H^1(JC_1\times JC_2,\Z)= H^2(JC_1\times JC_2,\Z) .
$$
The compatibility in \eqref{eq:cup-product} says that it restricts to the principal polarization on $JX\cong \im(f^\vee)$.
That is: 
\begin{align} \label{eq:quasi-polarization-preserved}
(f^\vee)^\ast \pr_1^\ast[\Theta_{C_1}]- (f^\vee)^\ast \pr_2^\ast[\Theta_{C_2}]=[\Theta_X]\in H^2(JX,\Z).
\end{align}
If $i=2c$, then the isomorphism $H^{2c}(A,\Z)\cong H_{2c}(A,\Z)$ from \eqref{eq:H^i(A)=H_i(A)} sends $[\Theta]^c/c!$ to $[\Theta]^{g-c}/(g-c)!$, where $g=\dim A$. 
Therefore, \eqref{eq:quasi-polarization-preserved} is equivalent to 
$$
f_\ast[C_1\times 0]-f_\ast[0\times C_2]=[\Theta_X]^{g-1}/(g-1)! \in H_2(JX,\Z) .
$$
This proves item \ref{item:thm:min-class:1} in the theorem, and thus finishes the proof. 
\end{proof}

\begin{remark}
Beckmann and de Gaay Fortman \cite{BGF} proved the integral Hodge conjecture for curve classes on products of Jacobians of curves. 
Thus item~\ref{item:thm:min-class:1} in Theorem \ref{thm:coho-dec-diagonal-complex-threefolds} may in fact be seen as a general consequence of item \ref{item:thm:min-class:2}. 
Using this, Voisin \cite[Theorem~1.4]{voisin-min-class} showed that the integral Hodge conjecture for curve classes on an abelian variety $A$ holds precisely when $A$ is a summand of a product of Jacobians.
By \cite{EGFS}, this fails for very general $(A,\Theta)$ of dimension at least four; partial results had previously been obtained in \cite{GFS}.
\end{remark}

\subsection{Regular matroids.} \label{subsec:matroids}
A \emph{matroid} $(\underline{R},S)$ (or simply $\underline{R}$) on a finite ground set $S$ is a collection of subsets of $S$, called \emph{independent sets}, abstracting the properties of linearly independent subsets of vectors in a vector space. 
The notion was introduced by Whitney 
in 1935, see \cite{oxley} for more details.

A basic example arises from a map $S \to V$, $s \mapsto v_s$, where $V$ is a vector space over a field $k$. 
The independent sets are precisely those $I \subset S$ for which $\{v_s \mid s \in I\}$ is linearly independent in $V$. 
A matroid $(\underline{R},S)$ is said to be \emph{realizable} over $k$ if such a representation exists, and \emph{regular} if it is realizable over every field.

\begin{definition}
An \emph{integral realization} of a matroid $(\underline{R},S)$ is a map $S \to V$ into a free $\Z$-module $V$ such that for any field $k$, the induced map $S \to V \otimes_\Z k$ realizes $(\underline{R},S)$. 
\end{definition}

We frequently identify a realization $S \to V$ with the induced linear map $\Z^S \to V$. 
In Section \ref{subsec:cubic-threefolds}, the free $\Z$-module $V$ will often be the dual of another free $\Z$-module $U$.

Any matroid admitting an integral realization is necessarily regular. 
Conversely, every regular matroid admits an integral realization, because it can be realized by the column vectors of a totally unimodular matrix, i.e.\ a matrix all of whose minors lie in $\{0,\pm 1\}$ \cite[Theorem~6.6.3]{oxley}.

\begin{example}
{\rm 
Let $G$ be a finite connected graph with edge set $E$ and vertex set $V$. 
Fix an orientation of each edge;  
the simplicial chain complex of $G$ then takes the form
\[
0 \longrightarrow \Z^E \stackrel{\del}\longrightarrow \Z^V \longrightarrow 0 .
\]
Since $G$ is connected, the cokernel of $\del$ is free of rank $1$. 
Because the cohomology of this complex is unchanged under tensoring with any field $k$, the following constructions yield integral realizations, and hence regular matroids:
\begin{enumerate}[label={(\arabic*)}]
    \item The {\em graphic matroid} $M(G)$ is realized by
    $\Z^E \to \Z^E / H_1(G,\Z)$.
    Equivalently, one may take $\del\colon \Z^E \to U$, where $U \subset \Z^V$ is the kernel of the sum map $\Sigma\colon \Z^V \to \Z$.
    \item The {\em cographic matroid} $M^\ast(G)$ is realized by
    $\Z^E \to H_1(G,\Z)^\ast$, $e \mapsto x_e$,
    where $x_e\colon H_1(G,\Z)\to \Z$ is the restriction of the $e$-th coordinate projection $\Z^E \to \Z$ to $H_1(G,\Z)\subset \Z^E$.
\end{enumerate}
Both $M(G)$ and $M^\ast(G)$ are independent of the chosen edge orientations. 
}
\end{example}

A matroid $(\underline R,S)$ is called \emph{graphic} (resp.\ \emph{cographic}) if it is isomorphic to a graphic (resp.\ cographic) matroid. 
Tutte showed that  the graphic matroid $M(G)$ is isomorphic to a cographic matroid if and only if $G$ is planar, see \cite[Corollary~13.3.4]{oxley}. 

\begin{example}[Seymour--Bixby] \label{ex:R10}
The matroid $\underline R_{10}$ is defined by the columns of the totally unimodular matrix
\[
R_{10} = \begin{pmatrix} 
1 & 0 & 0 & 0 & 0 & -1 & 1 & 0 & 0 & 1 \\
0 & 1 & 0 & 0 & 0 & 1 & -1 & 1 & 0 & 0 \\
0 & 0 & 1 & 0 & 0 & 0 & 1 & -1 & 1 & 0 \\
0 & 0 & 0 & 1 & 0 & 0 & 0 & 1 & -1 & 1 \\
0 & 0 & 0 & 0 & 1 & 1 & 0 & 0 & 1 & -1 
 \end{pmatrix}.
\]
This matroid is regular, but neither graphic nor cographic. 
By Seymour's decomposition theorem \cite{seymour}, every regular matroid can be obtained (in a precise sense) from graphic and cographic matroids together with $\underline R_{10}$. 
\end{example}

\begin{example} \label{example:cographic}
Let $\pi\colon \mathcal C\to \Delta$ be a proper family of connected complex curves over the disc, with $\mathcal C$ regular, special fiber $C_0$ nodal with node set $S$, and general fiber $C_t$ smooth. 
The limit mixed Hodge structure on $H_1(C_t,\Z)$ gives a weight filtration $W_{-2}\subset W_{-1}\subset W_0=H_1(C_t,\Z)$. 
For each $s\in S$, choose a vanishing cycle $\gamma_s \in H_1(C_t,\Z)$, unique up to sign. 
Then $W_{-2}$ is generated by $\{\gamma_s\mid s\in S\}$ and we obtain a realization
\begin{align} \label{eq:realization-curves}
S \longrightarrow W_{-2}H_1(C_t,\Z),\quad s\mapsto \gamma_s .
\end{align} 
\end{example}

\begin{lemma} \label{lem:cographic-matroid}
The matroid with realization \eqref{eq:realization-curves} agrees with the cographic matroid of the dual complex $\Gamma(C_0)$.
\end{lemma}

\begin{proof}
The intersection pairing on $C_t$ identifies $W_{-2}H_1(C_t,\Z)$ with $({\rm gr}^W_0 H_1(C_t,\Z))^\ast$. 
Moreover,
\[
{\rm gr}^W_0 H_1(C_t,\Z) \cong H_1(\Gamma(C_0),\Z),
\]
where $\Gamma(C_0)$ is the dual complex of $C_0$. 
This follows from the contraction $C_t \to C_0$, the homotopy equivalence $\tilde C_0 \to C_0$ (where nodes are replaced by edges), and the natural map $\tilde C_0 \to \Gamma(C_0)$ where components are contracted to the respective vertices, see e.g.\ \cite[Proposition~5.10]{EGFS-survey}. 
Thus \eqref{eq:realization-curves} identifies with
\[
S \longrightarrow H_1(\Gamma(C_0),\Z)^\ast,\quad s\mapsto x_s ,
\]
where $S$ corresponds to the edge set of $\Gamma(C_0)$ and $x_s$ is the restriction of the $s$-th coordinate function via the inclusion $H_1(\Gamma(C_0),\Z)\subset \Z^S$. 
(The dependence on orientations of the edges corresponds to choosing the sign of the vanishing cycle for each node.)
\end{proof}

\begin{theorem}[{Gwena \cite{gwena}}] \label{thm:Gwena}
Consider the Segre cubic threefold 
\begin{align} \label{eq:Segre-cubic}
Y_0\coloneq \left\{ \sum_{i=0}^5x_i=\sum_{i=0}^5x_i^3=0 \right\}\subset \CP^5_{\C}
\end{align} 
with $10$ ordinary double points.
Let $\mathcal Y\to \Delta$ be a smoothing of $Y_0$ over the unit disc $\Delta$ with regular total space $\mathcal Y$.
Let $S$ be the set of nodes of $Y_0$ and pick, for each $s\in S$, a vanishing cycle $\gamma_s\in H_3(Y_t,\Z)\cong H_1(JY_t,\Z)$ at a nearby fiber $Y_t$.
Then $S\to 
W_{-2}H_1(JY_t,\Z) $, $s\mapsto \gamma_s$,  
realizes the $\underline R_{10}$-matroid.
\end{theorem}

\begin{remark} \label{rem:bilinear-form}
Let $\pi\colon \mathcal A^\star\to \Delta^\star$ be a smooth projective family of principally polarised abelian varieties. 
Let $T\colon H_1(A_t,\Z)\to H_1(A_t,\Z)$ be the associated monodromy operator and assume that $N\coloneq T-\id$ is nilpotent. 
The associated monodromy bilinear form on ${\rm gr}^W_0H_1(A_t,\Z)$ is given by $x\otimes y\mapsto \Theta(Nx,y)$, where $\Theta$ denotes the bilinear pairing induced by the principal polarization on $A_t$, see \cite[Definition 2.12]{EGFS}.
If $\pi$ arises as the relative (intermediate) Jacobian of the example in Example \ref{example:cographic} (resp.\ Theorem \ref{thm:Gwena}), then the Picard--Lefschetz formula shows $B=\sum_{s\in S}x_s^2$, where $S$ is the set of nodes and $x_s\coloneq \Theta(\gamma_s,-)$ denotes the linear form on ${\rm gr}^W_0H_1(A_t,\Z)$ induced by the vanishing cycle $\gamma_s$ at $s$ and the polarization.
One can show that $B$ is positive definite and so the Delaunay decomposition (\cite[Definition 6.1]{EGFS-survey}) with respect to $B$ allows to recover the vanishing cycles (or rather the associated hyperplane arrangement given by $\{x_s=0\}_{s\in S}$ inside ${\rm gr}^W_0H_1(A_t,\R)$).
\end{remark}

\begin{corollary}[Gwena \cite{gwena}] \label{cor:gwena}
    The very general cubic threefold $Y\subset \CP^4_\C$ is irrational.
\end{corollary}
\begin{proof}
We use that a rational smooth complex projective threefold $Y$ satisfies $(JY,\Theta_Y)\cong (JC,\Theta_C)$ for some smooth (possibly disconnected) curve $C$, cf.\ \cite{clemens-griffiths}.
Consider the family $\pi_1 \colon \mathcal Y\to \Delta$ from Theorem \ref{thm:Gwena} and assume that its very general fiber is rational.
Then,  up to a finite base change, we may assume that there is a family of curves $\pi_2\colon \mathcal C\to \Delta$ together with fiberwise isomorphisms $JY_t\cong JC_t$ for all $t\neq 0$.
This isomorphism preserves the respective principal polarizations.
Therefore, the monodromy bilinear forms associated to the restrictions of $\pi_1$ and $\pi_2$ to $\Delta^\star$ coincide, see
Remark \ref{rem:bilinear-form}.
Comparing the respective Delaunay decompositions, we find that the hyperplane arrangements induced by the vanishing cycles on both families agree and so do the respective matroids.
By Lemma \ref{lem:cographic-matroid} and Theorem \ref{thm:Gwena}, we conclude that $\underline R_{10}$ is cographic, which is a contradiction.
\end{proof}

\subsection{Cubic threefolds.} \label{subsec:cubic-threefolds}
In joint work with Engel and de Gaay Fortman \cite{EGFS-survey,EGFS}, we proved:

\begin{theorem}[\cite{EGFS}] \label{thm:EGFS}
Let $Y\subset \CP^4_\C$ be a very general cubic threefold.
Then the minimal curve class $[\Theta_Y]^4/4!\in H_2(JY,\Z)$ is not algebraic.
\end{theorem}

By Voisin's Theorem \ref{thm:coho-dec-diagonal-complex-threefolds}, we obtain the following, which implies Theorem \ref{thm:EGFS-intro} from the introduction.

\begin{corollary}
A very general cubic threefold $Y\subset \CP^4_\C$ does not admit a cohomological decomposition of the diagonal.
\end{corollary}

The intermediate Jacobian $JY$ of a cubic threefold is a Prym variety, so $2\cdot [\Theta_Y]^4/4!$ is algebraic.  
For very general $Y$, every curve $C\subset JY$ has class $[C]=m[\Theta]^4/4!$ and it suffices to show that $m$ is always even.
We then consider a morphism $f\colon JC\to JY$ from the Jacobian of a smooth (not necessarily connected) curve $C$ such that
\[
f_\ast[C]=m[\Theta]^4/4!\in H_2(JY,\Z).
\] 
Assume for a contradiction that $m$ is odd.  
The dual map $f^\vee\colon JY\to JC$, defined using the principal polarizations, then satisfies
$$
(f^\vee)^\ast \Theta_C=m\Theta_Y \quad \quad \text{and} \quad \quad f\circ f^\vee=m\cdot \id_{JY} .
$$ 
Localizing at $2$, we let $\Lambda=\Z_{(2)}$. Since $m$ is invertible in $\Lambda$, $\frac{1}{m}f^\vee_\ast$ splits $f_\ast$, and we get a decomposition
\begin{align} \label{eq:H_1(JC)-decomposition}
H_1(JC,\Lambda)\cong H_1(JY,\Lambda)\oplus \ker(f_\ast\colon H_1(JC,\Lambda)\to H_1(JY,\Lambda)).
\end{align}
Thus the $\Lambda$-homology of $JY$ is a direct summand of the $\Lambda$-homology of the Jacobian of a curve, cf.\ Theorem \ref{thm:coho-dec-diagonal-complex-threefolds}.

The rough idea is: if \eqref{eq:H_1(JC)-decomposition} holds for very general $Y$, then, over a suitable cover of the moduli space of cubics, it extends to families and hence it is preserved by monodromy.
In other words, it would, up to some multiples introduced by ramification indices, force the monodromy of cubic threefolds to appear as a summand of the monodromy of curves.
To make this precise, note that by Theorem \ref{thm:Gwena} and Remark \ref{rem:bilinear-form}, the local monodromy near the Segre cubic is governed by the $\underline R_{10}$-matroid.
Conversely, the monodromy of curves is by Lemma \ref{lem:cographic-matroid}  and Remark \ref{rem:bilinear-form} governed by cographic matroids. 
This motivates the following precise matroidal analogue of the splitting in \eqref{eq:H_1(JC)-decomposition}; for more details on how to arrive there, see \cite[\S 1.4 and \S 2.5]{EGFS}.

\begin{definition}[{\cite[Definition 1.7]{EGFS}}] \label{def:d-Lambda-splitting-general}
Let $(\underline R,S)$ and $(\underline M,E)$ be regular matroids with integral realizations $S\to U^\ast$, $s\mapsto y_s$, and $E\to V^\ast$, $e\mapsto x_e$.  
Let $\Lambda$ be a ring and $d>0$.  
A \emph{quadratic $\Lambda$-splitting of level $d$} of $(\underline R,S)$ in $(\underline M,E)$ consists of an embedding $U_\Lambda\hookrightarrow V_\Lambda$ and a decomposition
\[
V_\Lambda=U_\Lambda\oplus U',\quad \text{
for some $U'\subset V_\Lambda$,}
\] 
together with a matrix $(a_{se})\in \Z_{\ge0}^{S\times E}$ such that $\forall s\in S$, the quadratic form $Q_s\coloneq \sum_{e}a_{se}x_e^2$ satisfies:
\begin{enumerate}[label={(\arabic*)}]
\item the above decomposition is orthogonal with respect to $Q_s$;
\item the restriction of $Q_s$ to $U_\Lambda$ equals $d\cdot y_s^2$.
\end{enumerate}
\end{definition}
 
The proof of Theorem \ref{thm:EGFS} follows from combining the next two theorems (see \cite[Theorems 1.8 and 1.9]{EGFS}), the first stated here only for intermediate Jacobians of cubics.

\begin{theorem}[\cite{EGFS}] \label{thm:algebraic->splitting}
Assume that the minimal class $[\Theta]^4/4!\in H_2(JY,\Z)$ on the intermediate Jacobian of a very general cubic threefold $Y\subset \CP^4_{\C}$ is algebraic.  
Then there exists $d\geq 1$ such that $\underline R_{10}$ admits a quadratic $\Z_{(2)}$-splitting of level $d$ into a cographic matroid.
\end{theorem}
 
\begin{theorem} [\cite{EGFS}] \label{thm:splitting->cographic}
A regular matroid $(\underline R,S)$ admits a quadratic $\Z_{(2)}$-splitting of some level $d\ge 1$ into a cographic matroid if and only if $(\underline R,S)$ is cographic.
\end{theorem}

To prove Theorem \ref{thm:algebraic->splitting}, we work with the universal deformation $\pi\colon \mathcal Y\to \mathrm{Def}_{Y_0}=\Delta^S$ of the Segre cubic threefold over a 10-dimensional polydisc (and not with $1$-parameter degenerations as in Corollary \ref{cor:gwena}).
The usage of degenerations over higher-dimensional bases distinguishes the argument in \cite{EGFS} from previous degeneration arguments in the subject, see e.g.\ \cite{kollar,voisin,CTP,Sch-duke}.
This leads to serious technical difficulties and we refer to \cite[\S 1.6]{EGFS} for an outline of the proof of both theorems.

Note that Theorem \ref{thm:splitting->cographic} is a purely combinatorial result; it is proven in \cite[§§5–7]{EGFS} via a sequence of reductions, exploiting ideas from algebraic geometry (Albanese varieties, see \cite[\S 5]{EGFS}) and topology (covering spaces, see \cite[\S 6]{EGFS}).
It eventually relies on a theorem of Tutte who showed that a regular matroid is cographic if and only if it does not contain the graphic matroids associated to the utility graph $K_{3,3}$ or the complete graph $K_5$  as a minor.
The case of $\underline R_{10}$ is covered by this, as it has $M(K_{3,3})$ as a minor.

\section*{Acknowledgments}
I am fortunate to work in a subject where many mathematicians have contributed to exciting advancements over the last decade and I thank all of them. 
I am grateful to J.\ Bowden, K.\ Hulek, J.C.\ Ottem, C.\ Voisin and O.\ Wittenberg for comments and discussions related to this text, 
and to J.-L.\ Colliot-Th\'el\`ene, P.\ Engel, O.\ de Gaay Fortman, and J.\ Lange for reading parts of this survey carefully. 
This project has received funding from the European Research Council (ERC) under the European Union’s Horizon 2020 research and innovation programme under grant agreement No 948066 (ERC-StG RationAlgic).

\end{document}